\documentclass[11pt,reqno]{amsart} 

\usepackage{lineno}

\usepackage{graphicx}

\usepackage{microtype}

\usepackage[top=1.25in,bottom=1.25in,left=1.25in,right=1.25in]{geometry} 

\usepackage{xcolor}

\usepackage{enumitem}

\usepackage[labelfont=bf]{caption}

\usepackage[font={small},labelfont=md]{subfig}

\usepackage[noadjust,nosort]{cite}

\usepackage{amsmath, amsthm, amssymb, tikz, mathtools, array, mathrsfs, tensor}

\usepackage{polynom}

\usepackage{esint}

\usepackage{multicol}

\usepackage{dsfont}
	\newcommand{\one}{\mathds{1}}

\usepackage{framed}

\usepackage{stackengine}

\usepackage{upref}

\usepackage{hyperref}
	\hypersetup{colorlinks=true,linkcolor=blue,citecolor=blue,urlcolor=blue} 

\allowdisplaybreaks

\numberwithin{equation}{section}


\newcommand{\eq}[1]{\begin{align*} #1 \end{align*}}
\newcommand{\eeq}[1]{\begin{align} \begin{split} #1 \end{split} \end{align}}

\newcommand{\stackref}[2]{\stackrel{\mbox{\footnotesize{\eqref{#1}}}}{#2}}
\newcommand{\stackrefp}[2]{\stackrel{\phantom{\mbox{\footnotesize{\eqref{#1}}}}}{#2}}

\newcommand{\G}{\mathbb{G}}

\renewcommand{\P}{\mathbb{P}}

\renewcommand{\AA}{\mathcal{A}}
\newcommand{\BB}{\mathcal{B}}

\newcommand{\EE}{\mathcal{E}}

\newcommand{\NN}{\mathcal{N}}

\newcommand{\PP}{\mathcal{P}}

\newcommand{\wt}[1]{\widetilde{#1}}

\newcommand{\givenp}[3][]{#1( #2 \: #1| \: #3 #1)} 

\newcommand{\cc}{\mathrm{c}} 

            \makeatletter
            \DeclareFontFamily{OMX}{MnSymbolE}{}
            \DeclareSymbolFont{MnLargeSymbols}{OMX}{MnSymbolE}{m}{n}
            \SetSymbolFont{MnLargeSymbols}{bold}{OMX}{MnSymbolE}{b}{n}
            \DeclareFontShape{OMX}{MnSymbolE}{m}{n}{
                <-6>  MnSymbolE5
               <6-7>  MnSymbolE6
               <7-8>  MnSymbolE7
               <8-9>  MnSymbolE8
               <9-10> MnSymbolE9
              <10-12> MnSymbolE10
              <12->   MnSymbolE12
            }{}
            \DeclareFontShape{OMX}{MnSymbolE}{b}{n}{
                <-6>  MnSymbolE-Bold5
               <6-7>  MnSymbolE-Bold6
               <7-8>  MnSymbolE-Bold7
               <8-9>  MnSymbolE-Bold8
               <9-10> MnSymbolE-Bold9
              <10-12> MnSymbolE-Bold10
              <12->   MnSymbolE-Bold12
            }{}
            
            \let\llangle\@undefined
            \let\rrangle\@undefined
            \DeclareMathDelimiter{\llangle}{\mathopen}%
                                 {MnLargeSymbols}{'164}{MnLargeSymbols}{'164}
            \DeclareMathDelimiter{\rrangle}{\mathclose}%
                                 {MnLargeSymbols}{'171}{MnLargeSymbols}{'171}
            \makeatother
            
    \DeclareFontFamily{U}{matha}{\hyphenchar\font45}
    \DeclareFontShape{U}{matha}{m}{n}{ <-6> matha5 <6-7> matha6 <7-8>
    matha7 <8-9> matha8 <9-10> matha9 <10-12> matha10 <12-> matha12 }{}
    \DeclareSymbolFont{matha}{U}{matha}{m}{n}
    \DeclareFontFamily{U}{mathx}{\hyphenchar\font45}
    \DeclareFontShape{U}{mathx}{m}{n}{ <-6> mathx5 <6-7> mathx6 <7-8>
    mathx7 <8-9> mathx8 <9-10> mathx9 <10-12> mathx10 <12-> mathx12 }{}
    \DeclareSymbolFont{mathx}{U}{mathx}{m}{n}
    
    \DeclareMathDelimiter{\llbrack} {4}{matha}{"76}{mathx}{"30}
    \DeclareMathDelimiter{\rrbrack} {5}{matha}{"77}{mathx}{"38}
  
    \makeatletter
    \newcounter{manualsubequation}
    \renewcommand{\themanualsubequation}{\alph{manualsubequation}}
    \newcommand{\startsubequation}{%
      \setcounter{manualsubequation}{0}%
     ~\refstepcounter{equation}\ltx@label{manualsubeq\theequation}%
      \xdef\labelfor@subeq{manualsubeq\theequation}%
    }
    \newcommand{\tagsubequation}{%
      \stepcounter{manualsubequation}%
      \tag{\ref*{\labelfor@subeq}\themanualsubequation}%
    }
    \let\subequationlabel\ltx@label
    \makeatother
            
\DeclareMathOperator{\cmp}{cmp}

\newtheorem{thm}{Theorem}[section]
\newtheorem{prop}[thm]{Proposition}
\newtheorem{cor}[thm]{Corollary}
\newtheorem{lemma}[thm]{Lemma}
\newtheorem{claim}[thm]{Claim}

\theoremstyle{definition}
\newtheorem{defn}[thm]{Definition}

\newtheorem{remark}[thm]{Remark}

\renewcommand{\thefootnote}{\fnsymbol{footnote}}

\title{Avoidance couplings on non-complete graphs}



\author{Erik Bates}
\address{\newline Department of Mathematics \newline University of Wisconsin--Madison \newline Van Vleck Hall \newline 480 Lincoln Drive \newline Madison, WI 53706-1324 \newline United States \newline \textup{\tt ewbates@wisc.edu}}

\author{Moumanti Podder}
\address{\newline NYU-ECNU Institute of Mathematical Sciences \newline New York University Shanghai \newline 1555 Century Avenue \newline Pudong, Shanghai 2001220 \newline People's Republic of China \newline \textup{\tt mp3460@nyu.edu}}

\begin{document}
\bibliographystyle{acm}

\renewcommand{\thefootnote}{\arabic{footnote}} \setcounter{footnote}{0}

\begin{abstract}
A coupling of random walkers on the same finite graph, who take turns sequentially, is said to be an {\it avoidance coupling} if the walkers never collide.
Previous studies of these processes have focused almost exclusively on complete graphs, in particular how many walkers an avoidance coupling can include.
For other graphs, apart from special cases, it has been unsettled whether even two non-colliding simple random walkers can be coupled.
In this article, we construct such a coupling on (i) any $d$-regular graph avoiding a fixed subgraph depending on $d$; and (ii) any square-free graph with minimum degree at least three.
A corollary of the first result is that a uniformly random regular graph on $n$ vertices admits an avoidance coupling with high probability.
\end{abstract}

\maketitle


\section{Introduction}
An avoidance coupling is a type of stochastic process introduced by Angel, Holroyd, Martin, Wilson, and Winkler \cite{angel-holroyd-martin-wilson-winkler13}.
Namely, it is a collection of simple random walkers on the same fixed graph, who take turns moving---one at a time and in cyclical order---yet no two of which ever occupy the same vertex.
It is a non-trivial matter to construct such a process, for despite this avoidance restriction we demand that each walker individually maintains the law of simple random walk.

In a broad view, there are two considerations governing the possibility that a given graph admits an avoidance coupling: its combinatorial features and its geometric features.
The former pertains to how the steps of the walkers can be coordinated so that each walker's marginal is faithful to simple random walk.
Meanwhile, the latter deals with the limitations imposed on this coordination by the presence or lack of particular edges in the graph.

The literature has, so far, mostly approached questions regarding only combinatorics and not geometry.
Beginning with \cite{angel-holroyd-martin-wilson-winkler13}, attention has given primarily to avoidance couplings on complete graphs.
Since all vertices are connected to all other vertices, and thus no walker is ever in a distinct geometric scenario, this case might be regarded as the `mean-field' model of avoidance couplings.
As the number of vertices tends to infinity, this symmetry enables the coordination of many walkers.
For instance, paired with \cite[Theorem 6.1]{angel-holroyd-martin-wilson-winkler13}, an article of Feldheim \cite{feldheim17} showed that the complete graph on $n$ vertices can accommodate an avoidance coupling of at least $\lceil n/4 \rceil$ walkers.
It is plausible that a linear upper bound of the form $cn$ with $c<1$ also holds (see \cite[Section 9]{angel-holroyd-martin-wilson-winkler13}), although the best available bound is $\lceil n-\log n \rceil$ by Bates and Sauermann \cite{bates-sauermann19}.

The only previous work to have considered non-complete graphs is the thesis of Infeld \cite{infeld16}.
In \cite[Chapter 2]{infeld16}, the question of existence is posed for avoidance couplings (of two walkers) possessing both a certain Markovian property and a certain uniformity property in their transition rates (see \cite[Definition 2.1]{infeld16}).
It was shown that this exceptional type of avoidance coupling, called a `uniform avoidance coupling', can be constructed on several families of graphs, including: cycles, bipartite graphs with minimum degree at least two, and strongly regular graphs in a certain parameter range.
Interestingly, other regular graphs and trees were proven to \textit{not} admit a uniform avoidance coupling, on the basis of their geometry.

\subsection{Main results}
The present article seeks to expand inquiry into the effect of geometry on the existence of avoidance couplings.
Our first main result is the following; the subgraphs mentioned in the theorem statement can be seen in Figure~\ref{bad_subgraphs}.

\begin{thm} \label{main_thm}
Assume $G$ is a $d$-regular graph on $n$ vertices, with $n\geq5$.

\begin{enumerate}[label=\textup{(\alph*)}]

\item \label{main_thm_a}
If $d\geq4$ and $G$ does not contain $H_d$ as a subgraph, then there exists an avoidance coupling of two walkers on $G$.

\item \label{main_thm_b}
If $d=3$ and $G$ does not contain $\wt H_3$ as a subgraph, then there exists an avoidance coupling of two walkers on $G$.

\end{enumerate}
\end{thm}

\begin{figure}
\centering
\subfloat[$H_d$]{\label{bad_d}
\includegraphics[width=0.5\textwidth]{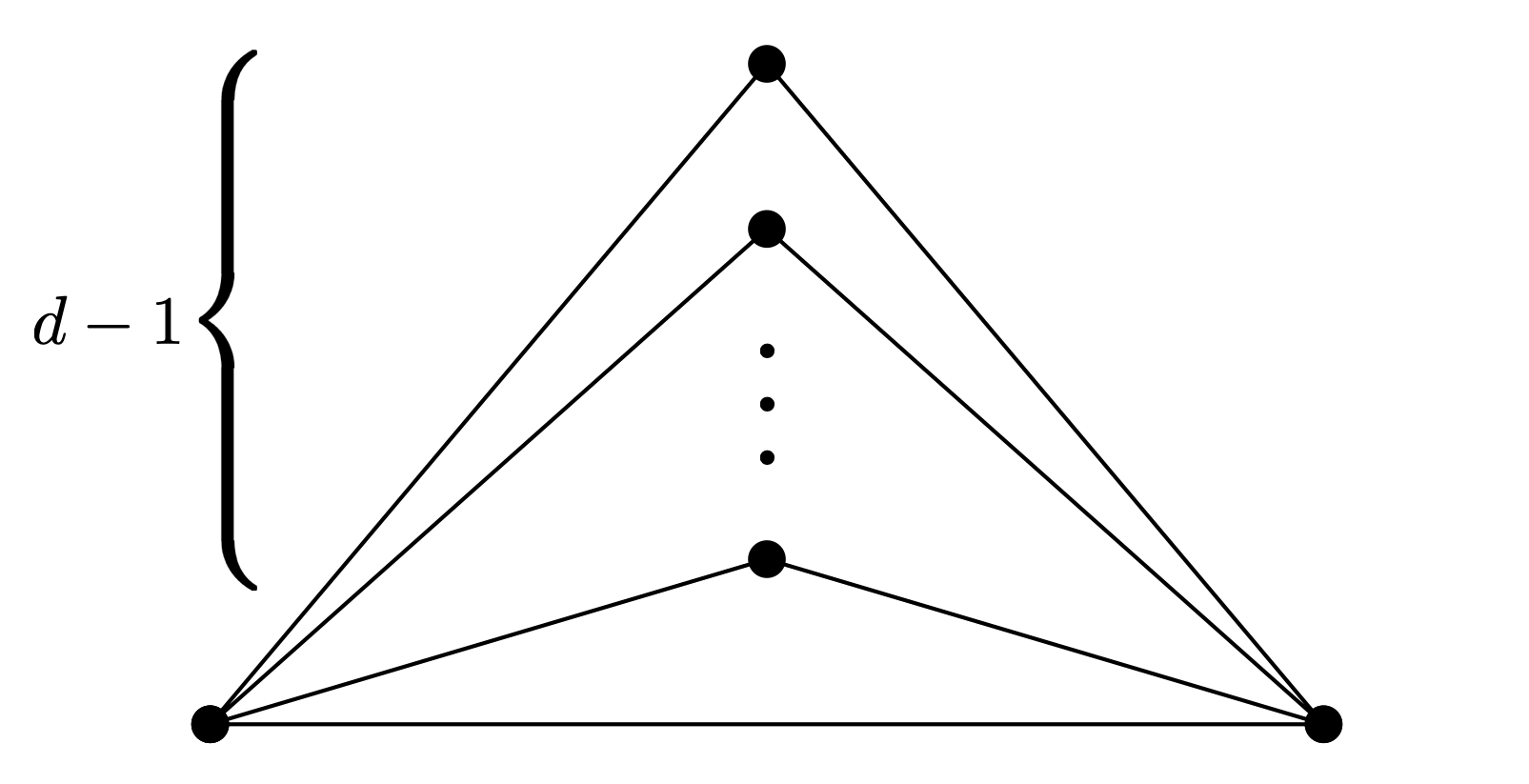}
}
\hspace{0.5in}
\subfloat[$\wt H_3$]{ \label{bad_3}
\includegraphics[width=0.3\textwidth]{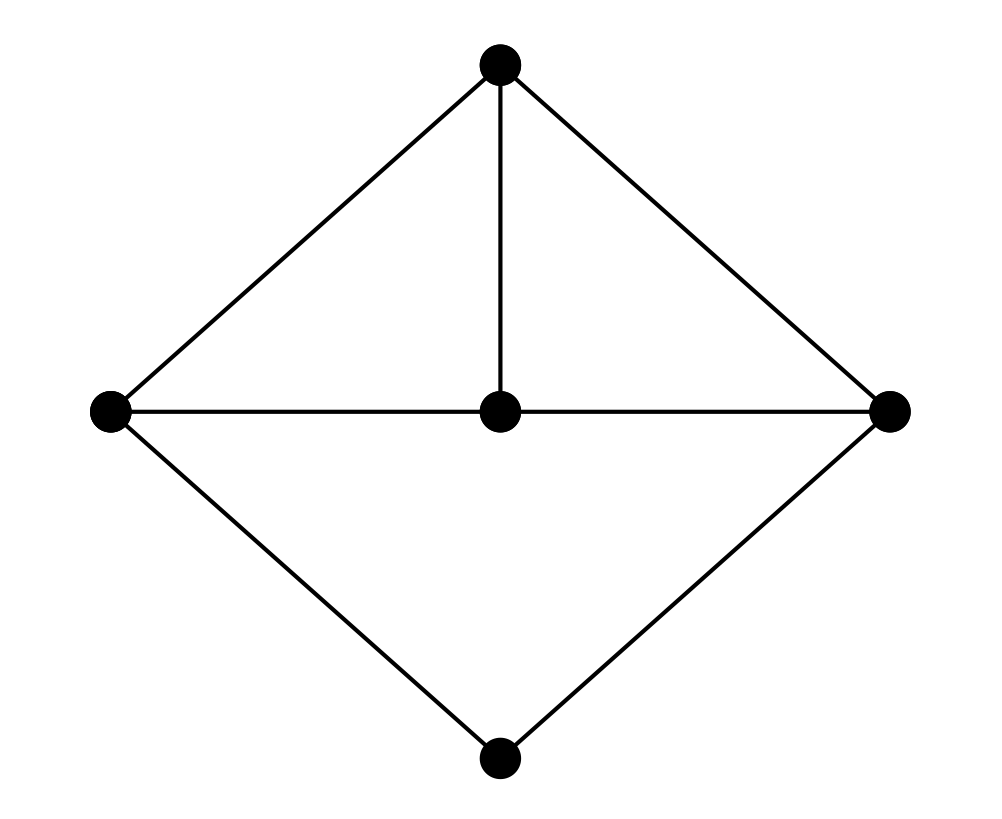}
}
\caption{Subgraphs to be avoided for Theorems~\ref{main_thm}(a) and~\ref{main_thm}(b) to apply}
\label{bad_subgraphs}
\end{figure}

We remark that part \ref{main_thm_b} is slightly stronger than part \ref{main_thm_a}, since $\wt H_3$ contains $H_3$ as a subgraph.
These two results are proved in Sections~\ref{3_regular} and~\ref{d_regular}, respectively.
While we are careful to verify the delicate technical details, high-level summaries of our constructions can be read in Sections~\ref{3_outline} and~\ref{construction_outline}.
While it would be interesting to know if the subgraph restrictions in Theorem~\ref{main_thm} can be relaxed, they are sufficiently loose to give the following corollary, proved in Appendix~\ref{appendix}.

\begin{cor} \label{main_cor}
Let $G_n$ be a random connected $d$-regular graph, uniform among those on $n$ vertices.
Then
\eq{
\lim_{n\to\infty}\P(\text{$G_n$ admits an avoidance coupling of two walkers}) = 1.
}
\end{cor}

This corollary is noteworthy for at least two reasons: 
(i) it is the first result concerning avoidance couplings on \textit{random} graphs; and (ii) a well-known fact is that random regular graphs are, with high probability, expander graphs.
In turn, expanders are frequently used in computer science applications, which serve as motivation for studying avoidance couplings in the first place (see Section~\ref{background}).
While it is true that these applications often ask for a \textit{non}-random expander graph, Corollary~\ref{main_cor} provides a basis for what can be expected of a well-constructed deterministic one.

\begin{remark}
When $d=2$, there is only one possible graph: $G$ is the cycle $C_n$ on $n$ vertices.
For this special case, it is easy to see that not just two walkers, but $k$ walkers for any $k\leq n/2$, can be avoidance coupled on $C_n$.
Indeed, label the vertices $0,1,\dots,n-1$ in the natural order, and assume $k$ walkers are initialized at positions $0,2,\dots,2(k-1)$.
In each unit of time, all walkers move in the same direction, either $+1$ or $-1$ modulo $n$, each case occurring with probability $1/2$ (independently between rounds).
It is then clear that each walker performs simple random walk on the cycle, and that the distance between any two walkers remains the same at the end of each time step.
Since these pairwise distances are all at least $2$, no two walkers will ever intersect.
\end{remark}

The second main result extends our constructions beyond regular graphs by capitalizing on the some of the methods developed for Theorem~\ref{main_thm}.
The proof appears in Section~\ref{square_free}, including a summary provided in Section~\ref{construction_square_free}.
Recall that a graph $G = (V,E)$ is \textit{square-free} if  whenever $v_1,v_2,v_3,v_4\in V$ are distinct, at least one of the edges $\{v_1,v_2\},\{v_2,v_3\},\{v_3,v_4\}$, $\{v_4,v_1\}$ is not found in $E$.

\begin{thm} \label{square_free_thm}
If $G$ is a square-free graph with minimum degree at least $3$, then there exists an avoidance coupling of two walkers on $G$.
\end{thm}

\subsection{Applications and related literature} \label{background}

Coupling of Markov chains has been a central tool in probability for many years.
Perhaps the most well-known modern example is in the study of mixing times \cite{levin-peres-wilmer09,levin-peres17}, although in that setting the desired coupling is one bringing different trajectories together.
Namely, a coupling achieving coincidence quickly can provide a theoretical guarantee of rapid mixing, information much wanted by practitioners of Markov chain Monte Carlo, for example.
In contrast, here we seek to keep two random walkers apart, a scenario that is also rooted in application, for instance within computer science, communication theory, and polling design (see \cite[Section 2.2]{bates-sauermann19} and references therein).
Also appearing at the interface of computer science and probability is the related class of `scheduling problems' \cite{coppersmith-tetali-winkler93,tetali-winkler93,gacs11,basu-sidoravicius-sly19}, which possess rich structure connected to dependent percolation models \cite{balister-bollobas-stacey00,winkler00,moseman-winkler08,pete08,brightwell-winkler09,gacs10,grimmett10}.

Given the applications of avoidance couplings, it is worth mentioning that all constructions in this paper are completely local in nature.
That is, the trajectories of the walkers only become dependent when the walkers are within a certain fixed distance of each other.
Therefore, the computational burden of the avoidance protocols is very small and thus conducive to real-world implementation.

\subsection{Notation}
Before proving the main results, let us set some notation that will be used throughout the paper.
We will always write $G = (V,E)$ to denote a finite simple graph.
For each vertex $v\in V$, we denote the neighborhood of $v$ by
\eq{
N(v) \coloneqq \{u \in V : \{u,v\}\in E\}.
}
The degree of a vertex will be written $\deg(v) \coloneqq |N(v)|$.
Recall that a (discrete-time) simple random walk on $G$ is a $V$-valued Markov chain $(X_t)_{t\geq0}$ such that
\eq{
\P\givenp{X_{t+1}=u}{X_t=v} = 
\begin{cases} 1/\deg(v) &\text{if }u\in N(v), \\
0 &\text{otherwise}.
\end{cases}
}
We can now make precise the notion of an avoidance coupling, hereafter of two walkers unless otherwise noted.

\begin{defn} \label{ac_def}
An \textit{avoidance coupling} on $G$ is a discrete-time ($V\times V$)-valued process $(A_t,B_t)_{t\geq0}$ such that
\begin{enumerate}[label=\textup{\arabic*.}]
\item (\textit{faithfulness}) $(A_t)_{t\geq0}$ and $(B_t)_{t\geq0}$ are each simple random walks on $G$; and
\item (\textit{avoidance}) $B_t\notin\{A_t,A_{t+1}\}$ for all $t\geq0$.
\end{enumerate}
\end{defn}

For the curious reader, we mention that continuous-time avoidance couplings do not exist on connected graphs, by \cite[Theorem 2.1]{angel-holroyd-martin-wilson-winkler13}.
Moreover, two walkers initialized in different components of a graph trivially avoid one another at all times.
Therefore, we assume henceforth that $G$ is connected.
Also, for brevity, we will occasionally write $A_{[0,t]}$ to denote the vector $(A_0,A_1,\dots,A_t)$.

%
%

\section{Proof of Theorem~\hyperref[main_thm_b]{\ref*{main_thm}\ref*{main_thm_b}}: $3$-regular graphs} \label{3_regular}

\setcounter{subsection}{-1}
\subsection{Outline of coupling} \label{3_outline}

Absent the special symmetry of the $2$-regular case, avoidance couplings on $3$-regular graphs must accommodate many more situations in which collision could occur.
The strategy of this section is to group these situations into finitely many cases, and then address each case separately.
Once this is done, an avoidance coupling $(A_t,B_t)_{t\geq0}$ is produced inductively as follows.
In the spirit of \cite{angel-holroyd-martin-wilson-winkler13}, our two walkers shall be named Alice and Bob.

\begin{itemize}
\item At time $t\geq0$, suppose Alice is at vertex $A_t = a$ and is next to move.
\item Assume Bob is at vertex $B_t = b$, which is neither equal to $a$ nor a neighbor of $a$.
\item Depending on the graph's local structure around $a$ and $b$, the two walkers agree to jointly specify their next $T\geq1$ steps. 
Here $T$ may be random, and the two walkers' coordination will be such that:
\begin{enumerate}[label=\textup{(\roman*)}]
\item \label{condition_1}
each step is uniformly random, meaning that for each $s \in\{0,1,\dots, T-1\}$,
\begin{subequations}
\label{srw_condition}
\begin{align}
\label{srw_condition_a}
\P\givenp{A_{t+s+1}=a'}{A_{[0,t+s]}} &= \frac{1}{|N(A_{t+s})|}\one_{\{a'\in N(A_{t+s})\}}, \\
\label{srw_condition_b}
\P\givenp{B_{t+s+1}=b'}{B_{[0,t+s]}} &= \frac{1}{|N(B_{t+s})|}\one_{\{b'\in N(B_{t+s})\}};
\end{align}
\end{subequations}
\item \label{condition_2} 
they never collide during these $T$ steps, meaning
\eq{
B_{t+s} \notin \{A_{t+s},A_{t+s+1}\}, \quad 0 \leq s \leq T-1 \qquad \text{and} \qquad
B_{t+T} \neq A_{t+T};
}
\item \label{condition_3}
Bob is not adjacent to Alice after $T$ steps, meaning $B_{t+T} \notin N(A_{t+T})$.
\end{enumerate}
\item Because of this last condition \hyperref[condition_3]{(iii)}, the routine can be executed \textit{ad infinitum}. 
Each iteration is independent of previous iterations, given Alice's and Bob's current locations.
\end{itemize}

\begin{prop} \label{conditions_to_coupling}
If the graph distance between $A_0$ and $B_0$ is at least $2$, then any process $(A_t,B_t)_{t\geq0}$ defined as above is an avoidance coupling on $G$.
\end{prop}

\begin{proof}
Let $T_1,T_2,\dots$ be the values of $T$ realized in the inductive algorithm, and set $S_\ell = T_1 + \cdots + T_\ell$.
The `avoidance' part of Definition~\ref{ac_def} is satisfied because of \hyperref[condition_2]{(ii)}.
The `faithfulness' requirement is equally trivial, but we note the following subtlety.
For any deterministic $t\geq0$, we wish to show that
\eq{
\P\givenp{A_{t+1}=a'}{A_{[0,t]}} = \frac{1}{|N(A_t)|}\one_{\{a'\in N(A_t)\}},
}
and likewise for Bob's trajectory.
The above condition \hyperref[condition_1]{(i)} actually shows
\eq{
\P\givenp{A_{t+1}=a'}{A_{[0,t]}}\one_{\{S_{\ell-1}\leq t<S_\ell\}} = \frac{1}{|N(A_t)|}\one_{\{a'\in N(A_t)\}}\one_{\{S_{\ell-1}\leq t<S_\ell\}},
}
but of course there is always some (unique) $\ell$ for which $S_{\ell-1}\leq t < S_\ell$.
Therefore, the latter display implies the former.
\end{proof}

For the remainder of Section~\ref{3_regular}, we assume $G$ is a $3$-regular graph on $n\geq5$ vertices, not containing $\wt H_3$ from Figure~\ref{bad_3} as a subgraph.
Note that $n\geq5$ implies $G$ is not equal to a complete graph.
Therefore, there do exist initial positions $A_0=a$ and $B_0=b$ such that $b \notin \{a\}\cup N(a)$, meaning Proposition~\ref{conditions_to_coupling} will prove Theorem~\ref{main_thm}(b) once a coupling satisfying  \hyperref[condition_1]{(i)}--\hyperref[condition_3]{(iii)} is exhibited.

Having determined the strategy, we consider in the coming sections various possibilities for the graph's local structure around $a$ and $b$.
In each case, we will precisely specify a `local coupling' and check that \hyperref[condition_1]{(i)}--\hyperref[condition_3]{(iii)} are satisfied.
The possibilities we list are both exhaustive (of pairs of vertices $a,b$ between which the graph distance at least $2$) and mutually exclusive, so that the general coupling prescribed above is well-defined.
Some of the scenarios will involve the following definition.

\begin{defn}
We say that the vertex pair $(a,b)$ \textit{admits} $K_{2,2}$ if there are four distinct vertices, namely $a_1,a_2\in N(a)$ and $b_1,b_2\in N(b)$, such that $a_1$ and $a_2$ are each adjacent to both $b_1$ and $b_2$.
(In other words, locally around $a$ and $b$, the graph $G$ is as in Figure~\ref{scenario_6}.)
Otherwise, we say $(a,b)$ \textit{does not admit} $K_{2,2}$.
\end{defn}

\subsection{Scenario 1: $a$ and $b$ at distance at least $4$} \label{case_1}
If the graph distance between $a$ and $b$ is at least $4$, then we take $T=1$ and allow
Alice and Bob to independently move to a uniformly random neighbor.
Hence condition \hyperref[condition_1]{(i)} is trivially satisfied.
After the transitions, the graph distance will still be at least $2$, and so \hyperref[condition_2]{(ii)} and \hyperref[condition_3]{(iii)} are also clear.

\subsection{Scenario 2: $a$ and $b$ have three common neighbors} \label{case_2}
Next suppose that $N(a) = N(b) = \{c_1,c_2,c_3\}$.
Observe that if $i\neq j$, then $c_i$ and $c_j$ are not adjacent, for otherwise $G$ would contain a copy of $\wt H_3$.
Figure~\ref{scenario_2} shows the resulting situation.
In particular, we can again take $T=1$ and let Alice and Bob \textit{jointly} select from one of the following sequences of transitions, with equal probability:
\eq{
A_{t+1}= c_1,\, B_{t+1}=c_2, \qquad
A_{t+1}=c_2,\, B_{t+1}=c_3, \qquad
A_{t+1}=c_3,\, B_{t+1}=c_1.
}
It is easy to see that this coupling satisfies conditions \hyperref[condition_1]{(i)}--\hyperref[condition_3]{(iii)}.

\begin{figure}
\subfloat[{\hyperref[case_2]{Scenario 2}}]{
\includegraphics[clip,trim=1in 0.6in 1in 0.7in,width=0.47\textwidth]{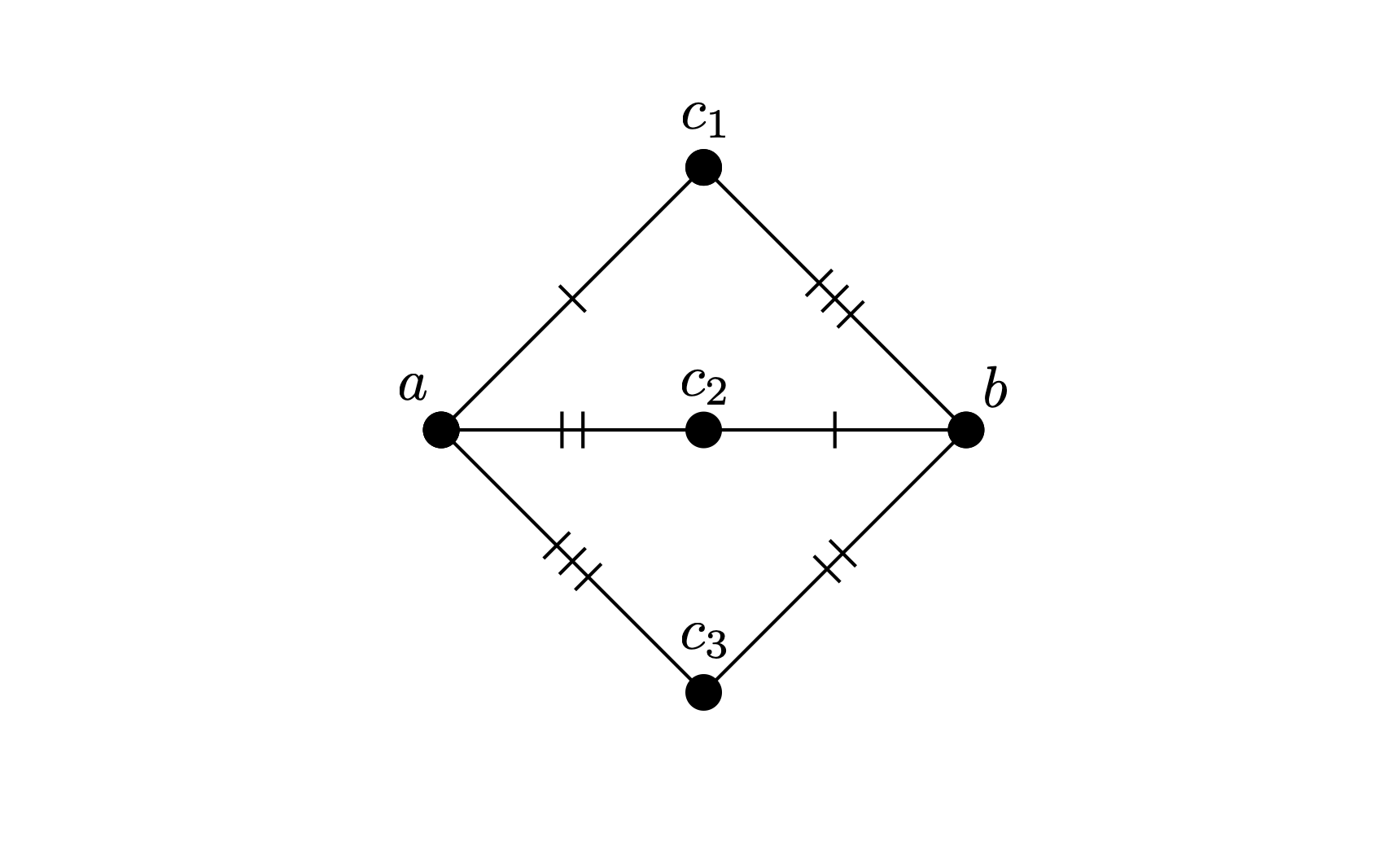}
\label{scenario_2}
}
\hfill
\subfloat[{\hyperref[case_3a]{Scenario 3a}}]{
\includegraphics[clip,trim=1in 0.6in 1in 0.7in,width=0.47\textwidth]{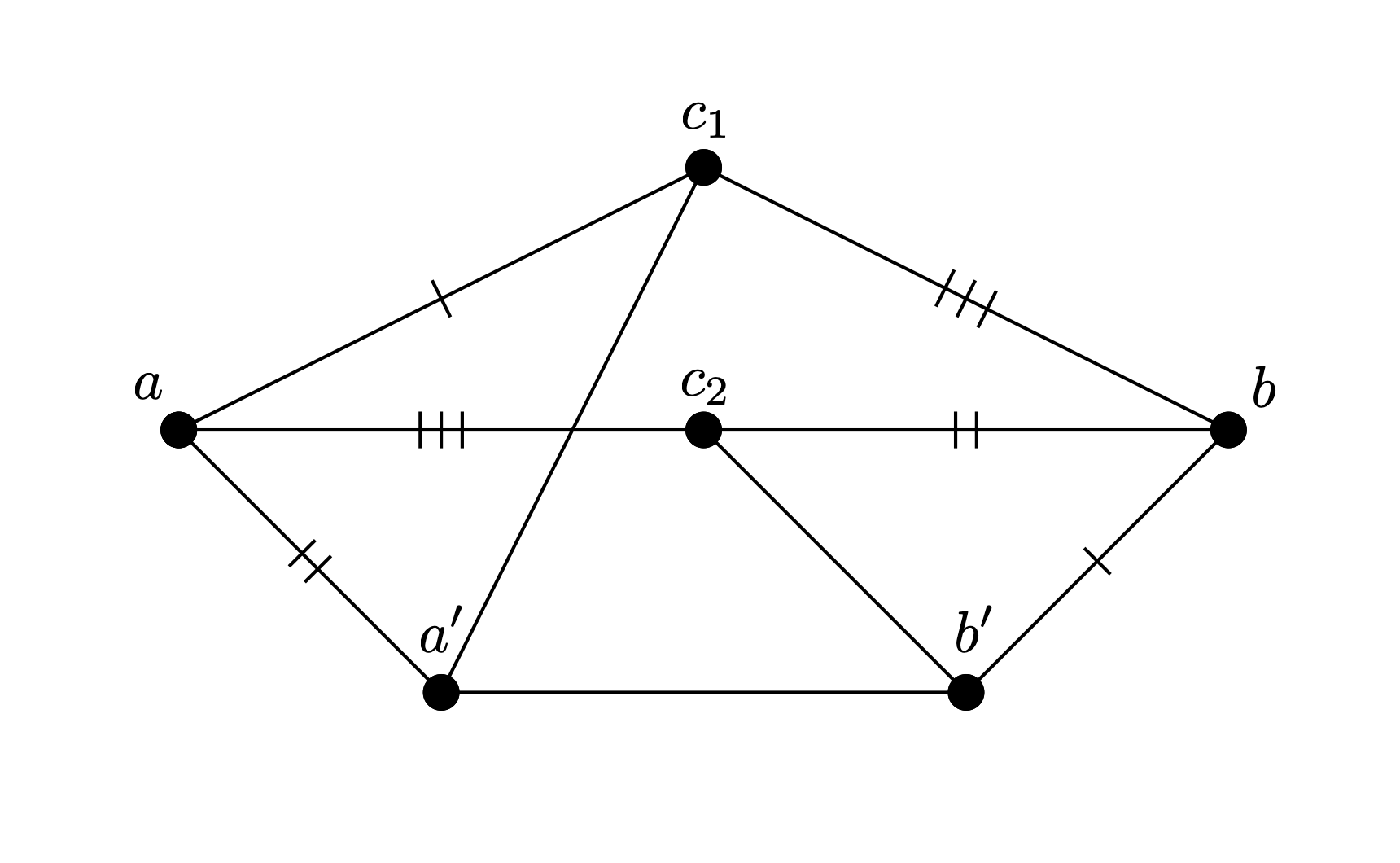}
\label{scenario_3a}
}
\\
\subfloat[{\hyperref[case_4]{Scenario 4}}]{
\includegraphics[clip,trim=1in 0.4in 1in 0.7in,width=0.47\textwidth]{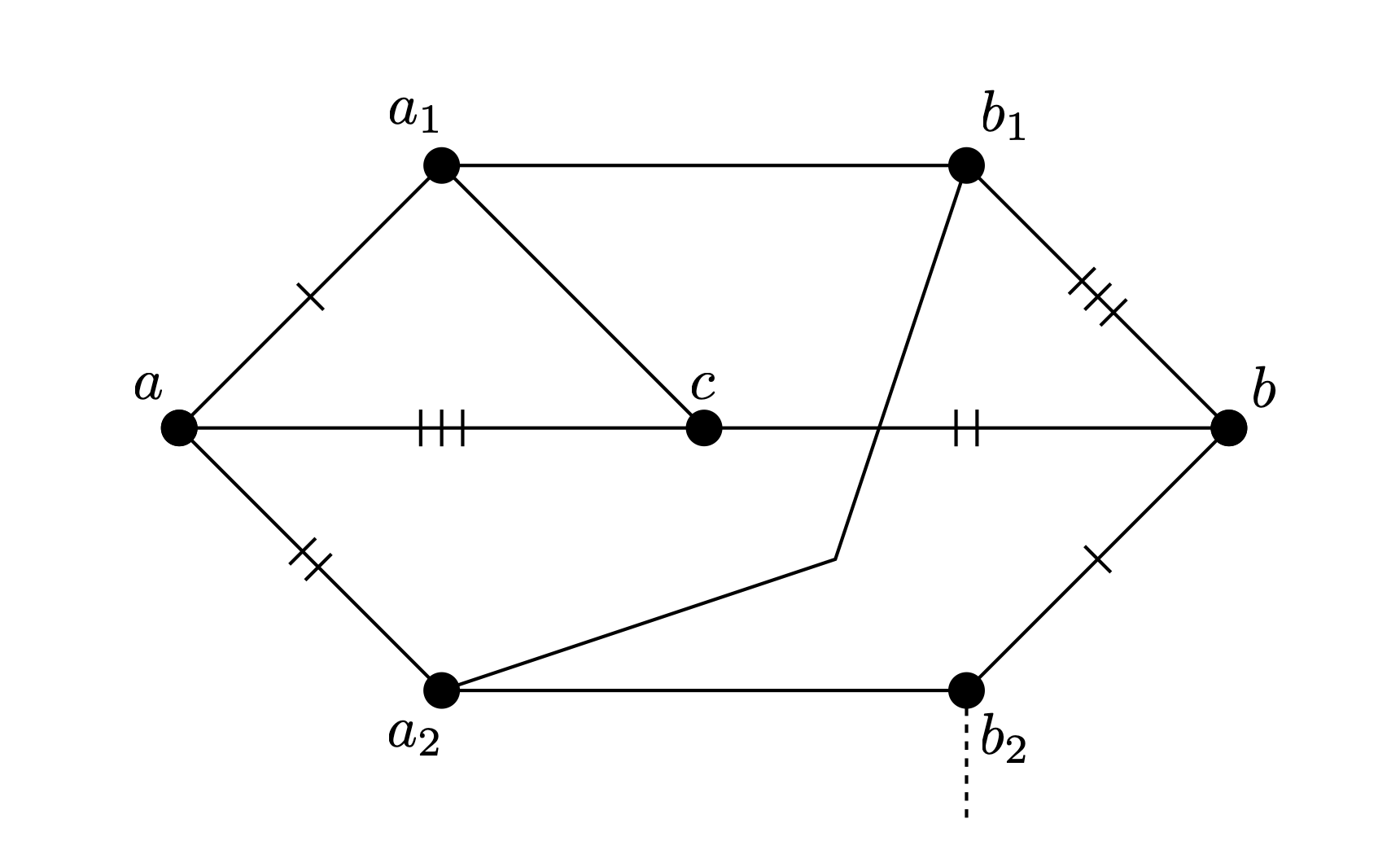}
\label{scenario_4}
}
\hfill
\subfloat[{\hyperref[case_5]{Scenario 5}}]{
\includegraphics[clip,trim=1in 0.4in 1in 0.7in,width=0.47\textwidth]{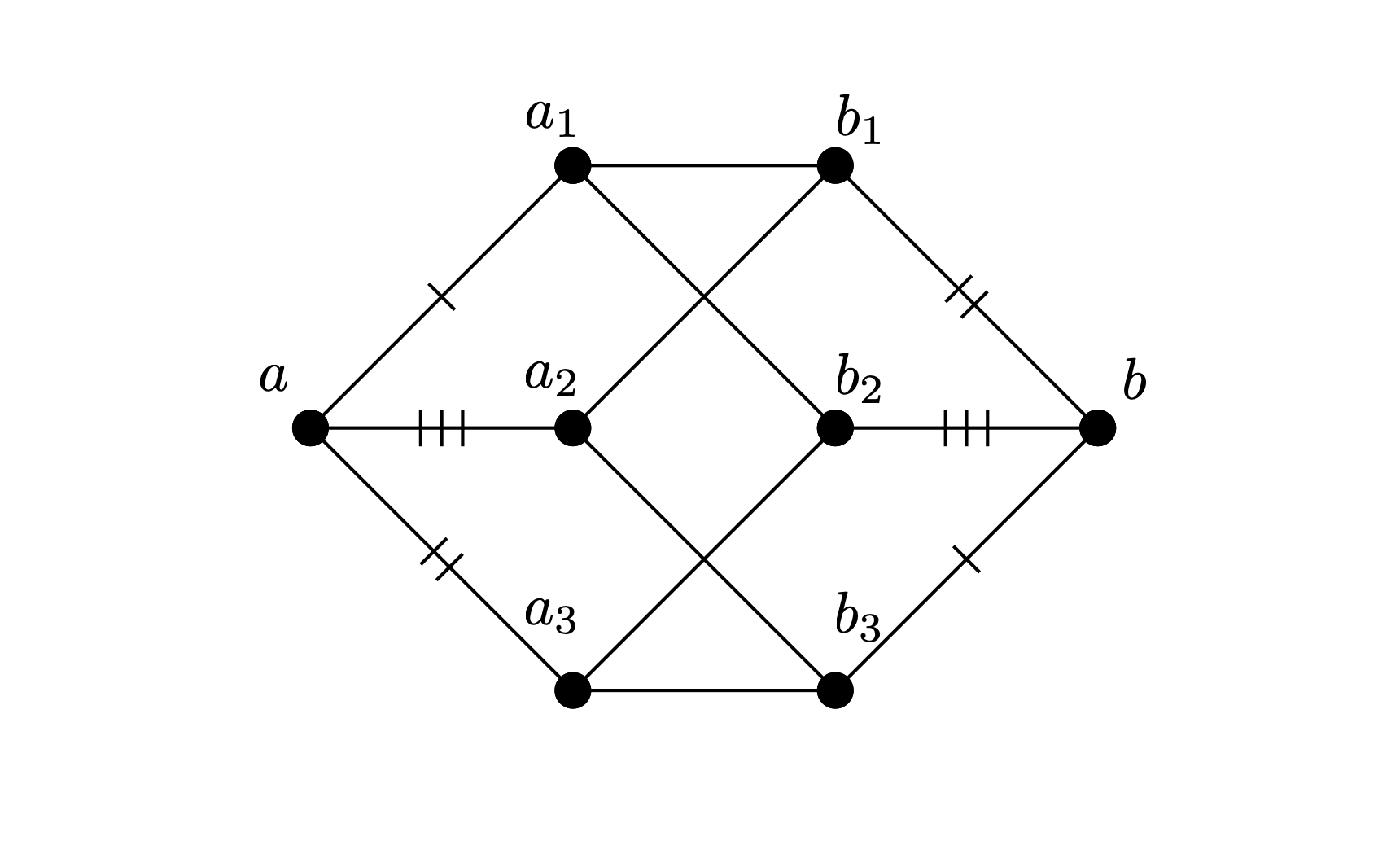}
\label{scenario_5}
}
\caption{Shown here are the local configurations in $3$-regular graphs for which a one-step coupling suffices.
In each case, the three possible edges to be traversed by Alice are matched with those to be traversed by Bob; corresponding edges are marked by the same number of dashes.
No matter which pair of edges is traversed, Alice and Bob will remain separated by graph distance at least 2 after having both moved.
In (a), each of $c_1,c_2,c_3$ is adjacent to one other vertex (not necessarily the same one) not shown in the diagram.
In (c), the same is true for $b_2$.
}
\label{one_step}
\end{figure}

\subsection{Scenario 3: $a$ and $b$ have two common neighbors} \label{case_3}
Let us write $N(a) = \{c_1,c_2,a'\}$ and $N(b) = \{c_1,c_2,b'\}$, where $a'\neq b'$.
We separately consider two cases.

\subsubsection{Scenario 3a: common neighbors $c_1,c_2$ are not adjacent} \label{case_3a}
As in Scenarios \hyperref[case_1]{1} and \hyperref[case_2]{2}, Alice and Bob will only need to choose their immediate next step in this situation; that is, $T=1$. 
Consequently, whether or not condition \hyperref[condition_3]{(iii)} is met
depends solely on the edges between neighbors of $a$ and neighbors of $b$.
Moreover, \hyperref[condition_3]{(iii)} is a strictly weaker requirement if any of these edges are removed.
Therefore, there is no loss of generality in assuming that $c_1$ is adjacent to $a'$, and $c_2$ is adjacent to $b'$; the resulting graph is displayed in Figure~\ref{scenario_3a}.
Alice and Bob thus select uniformly one of the following sequences of transitions:
\eq{
A_{t+1}= c_1,\, B_{t+1}=b', \qquad
A_{t+1}= a',\, B_{t+1}=c_2, \qquad
A_{t+1}= c_2,\, B_{t+1}= c_1.
}
Once again, this coupling is easily seen to satisfy conditions \hyperref[condition_1]{(i)}--\hyperref[condition_3]{(iii)}.

\subsubsection{Scenario 3b: common neighbors $c_1,c_2$ are adjacent} \label{case_3b}
In this case, we will need to take $T=2$.
Observe that $c_1$ and $c_2$ are each incident already to three edges, and so neither is adjacent to $a'$ or $b'$.
On the other hand, both $a'$ and $b'$ have two unspecified neighbors.
Let us write $N(a') = \{a,a''_1,a''_2\}$ and $N(b') = \{b,b''_1,b''_2\}$, where $b'$ may be equal to one of $a''_1,a''_2$ (equivalently, $a'$ may be equal to one of $b''_1,b''_2$).
In any case, however, we have
\eq{
\{a,c_1,c_2\} \cap N(b') = \varnothing = \{b,c_1,c_2\} \cap N(a').
}
Now Alice and Bob select uniformly one of the nine sequences of transitions listed in Figure~\ref{scenario_3b}.
Simple inspection shows that this two-step coupling satisfies conditions \hyperref[condition_1]{(i)}--\hyperref[condition_3]{(iii)}.

\begin{figure}
\includegraphics[clip,trim=1in 0.6in 1in 0.7in,width=0.47\textwidth]{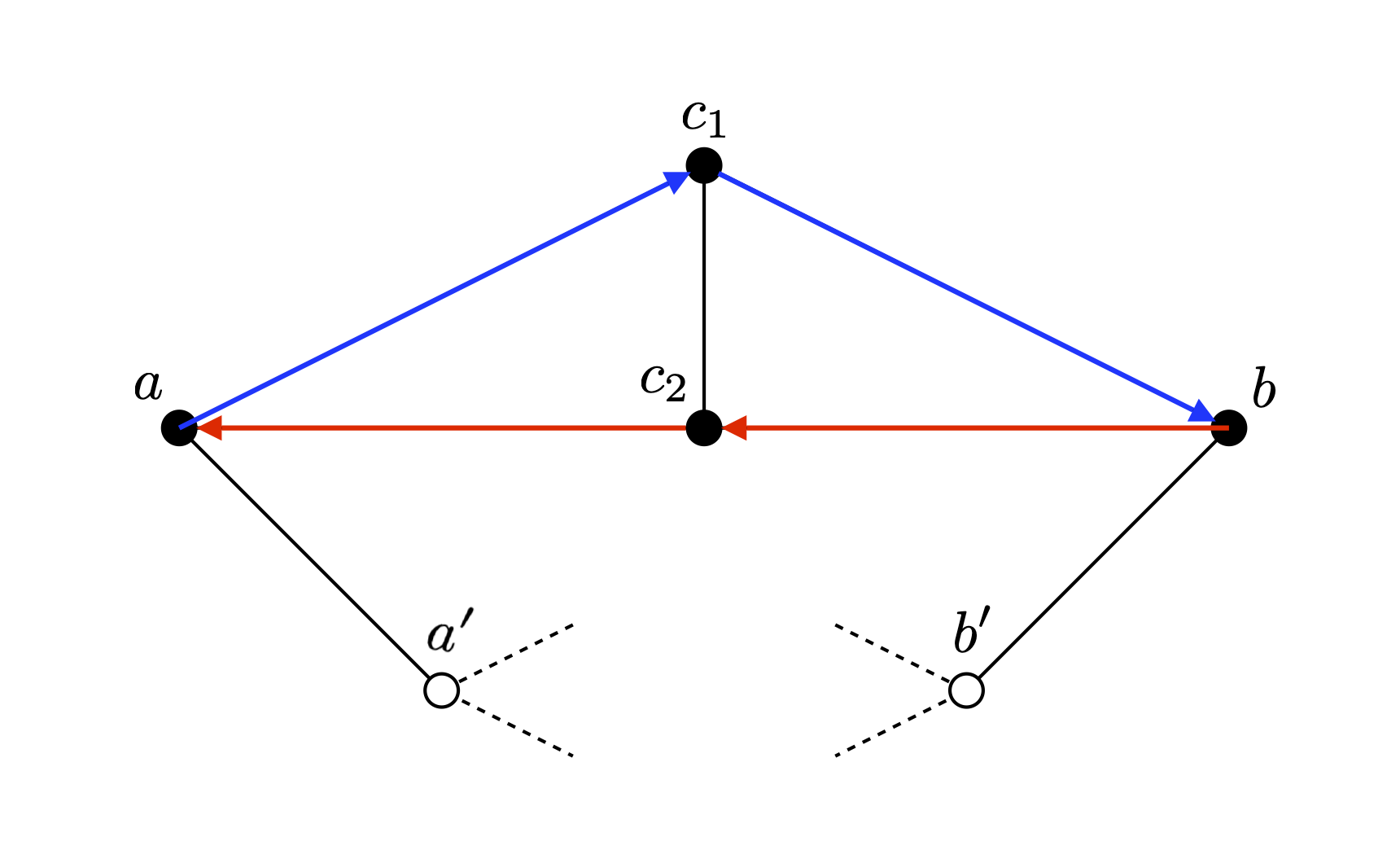}
\hspace{0.65in}
\begin{tabular}{cc|cc}
\vspace{-1.7in}
&&& \\
$A_{t+1}$ & $A_{t+2}$ & $B_{t+1}$ & $B_{t+2}$ \\ \hline
$c_1$ & $a$ & $c_2$ & $b$ \\
$c_1$ & $b$ & $c_2$ & $a$ \\
$c_1$ & $c_2$ & $b'$ & $b''_1$ \\
$c_2$ & $a$ & $c_1$ & $b$ \\
$c_2$ & $b$ & $c_1$ & $a$ \\
$c_2$ & $c_1$ & $b'$ & $b''_2$ \\
$a'$ & $a$ & $b'$ & $b$ \\
$a'$ & $a''_1$ & $c_1$ & $c_2$ \\
$a'$ & $a''_2$ & $c_2$ & $c_1$ \\
\end{tabular}
\caption{Shown here is \hyperref[case_3b]{Scenario 3b}, the local configuration in $3$-regular graphs for which a two-step coupling is required.
Alice and Bob jointly choose their next $T=2$ steps so as both to avoid collision and to remain distance at least $2$ apart after having each completed two transitions.
The paired movements are displayed in the table on the right.
As an example, the table's second row is illustrated in the diagram on the left.
Vertices $a'$ and $b'$ may be adjacent and/or share a neighbor, but the dashed edges in question---linking $a'$ to $a_1'',a_2''$ and $b'$ to $b_1',b_2''$---are not incident to any of the vertices shown as solid dots.
}
\label{scenario_3b}
\end{figure}

\subsection{Scenario 4: $a$ and $b$ have one common neighbor, $(a,b)$ does not admit $K_{2,2}$} \label{case_4}
In this circumstance, we return to needing only $T=1$.
Let us write $N(a) = \{c,a_1,a_2\}$ and $N(b) = \{c,b_1,b_2\}$, where $\{a_1,a_2\}\cap\{b_1,b_2\} = \varnothing$.
By assumption, at least one of the edges $\{a_1,b_1\}, \{a_1,b_2\}, \{a_2,b_1\}, \{a_2,b_2\}$ is not present.
Without loss of generality, we assume $\{a_1,b_2\}$ is not present.
By the same logic as in \hyperref[case_3a]{Scenario 3a}, because $T=1$, there is also no loss of generality in assuming that the remaining three edges are all present, as well as $\{a_1,c\}$.  (The scenario of $c$ being adjacent to $b_2$ is technically distinct because Alice will move from $a$ \textit{before} Bob moves from $b$, but because the two walkers start at distance $2$, the order is irrelevant).
As illustrated in Figure~\ref{scenario_4}, the following one-step coupling satisfies \hyperref[condition_1]{(i)}--\hyperref[condition_3]{(iii)}:
\eq{
A_{t+1}= a_1,\, B_{t+1}=b_2, \qquad
A_{t+1}= a_2,\, B_{t+1}=c, \qquad
A_{t+1}= c,\, B_{t+1}=b_1.
}

\subsection{Scenario 5: $a$ and $b$ at distance $3$, $(a,b)$ does not admit $K_{2,2}$} \label{case_5}
The situation here is similar to \hyperref[case_4]{Scenario 4}, and we will again take $T=1$.
Let us write $N(a) = \{a_1,a_2,a_3\}$ and $N(b) = \{b_1,b_2,b_3\}$, where $N(a)\cap N(b) = \varnothing$.
As before, the only relevant edges for condition \hyperref[condition_3]{(iii)} are those between $N(a)$ and $N(b)$.
Now, each $a_i$ is adjacent to at most two $b_j$'s. 
But by assumption, no two of the $a_i$'s are connected to the \textit{same} two $b_j$'s.
It is thus clear that the maximally restrictive graph is the one shown in Figure~\ref{scenario_5}, for which the suitable coupling is given by
\eq{
A_{t+1}= a_1,\, B_{t+1}=b_3, \qquad
A_{t+1}= a_2,\, B_{t+1}=b_2, \qquad
A_{t+1}= a_3,\, B_{t+1}=b_1.
}

\subsection{Scenario 6: $(a,b)$ admits $K_{2,2}$} \label{case_6}
This final scenario requires us to take a \textit{random} $T$.
Let us write $N(a) = \{a_1,a_2,c_a\}$ and $N(b) = \{b_1,b_2,c_b\}$, where $c_a$ and $c_b$ may be adjacent or even equal.
We are assuming that $a_1$ and $a_2$ are each adjacent to both $b_1$ and $b_2$, and hence all four of these vertices admit no further edges.
See Figure~\ref{scenario_6} for an illustration.
\begin{figure}
\subfloat[{\hyperref[case_6]{Scenario 6}}]{
\includegraphics[clip,trim=1in 2in 1in 0.7in,width=0.47\textwidth]{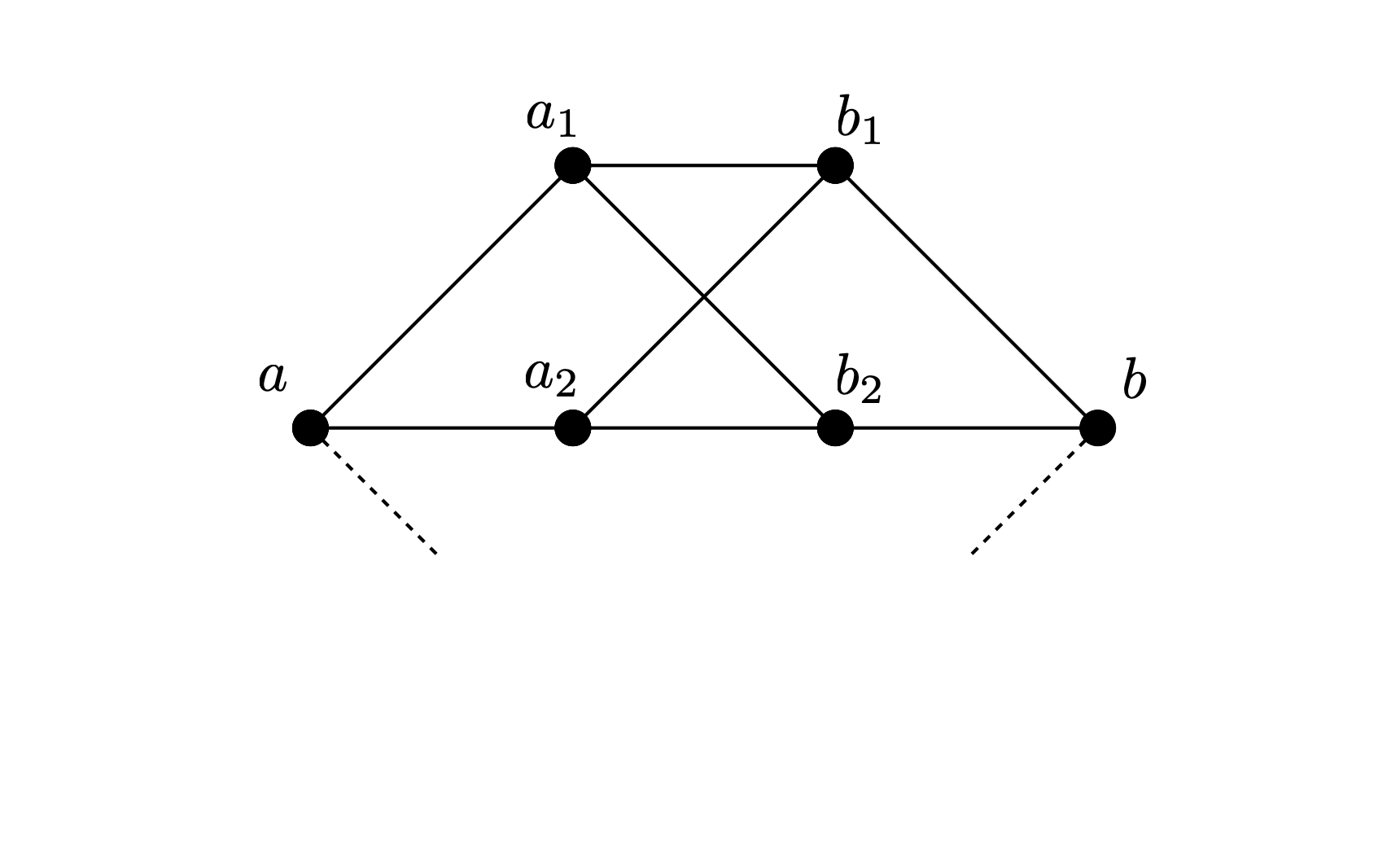}
\label{scenario_6}
}
\hfill
\subfloat[Coupling 1]{
\includegraphics[clip,trim=1in 2in 1in 0.7in,width=0.47\textwidth]{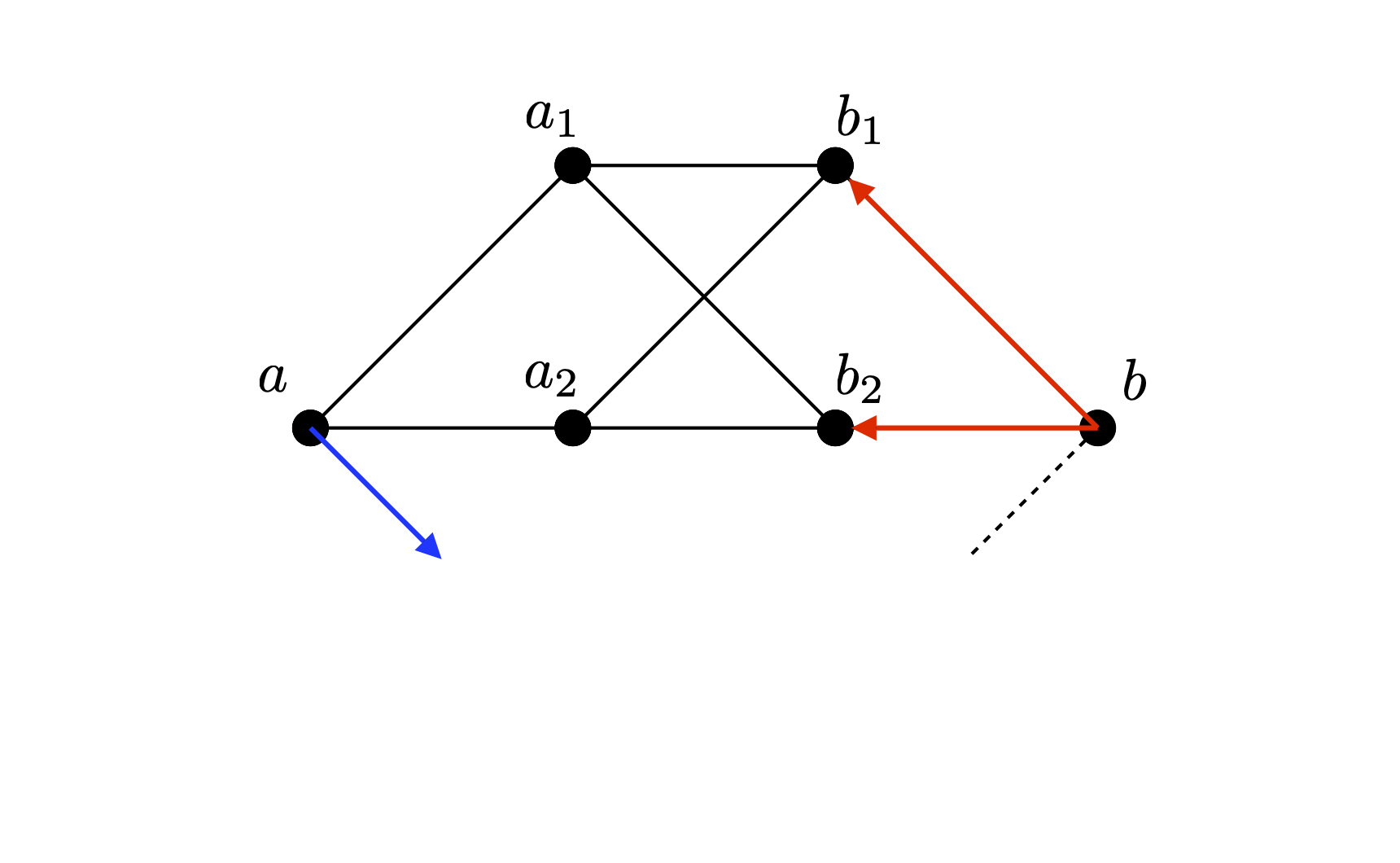}
\label{scenario_6_1}
}
\\
\subfloat[Coupling 2]{
\includegraphics[clip,trim=1in 2in 1in 0.7in,width=0.47\textwidth]{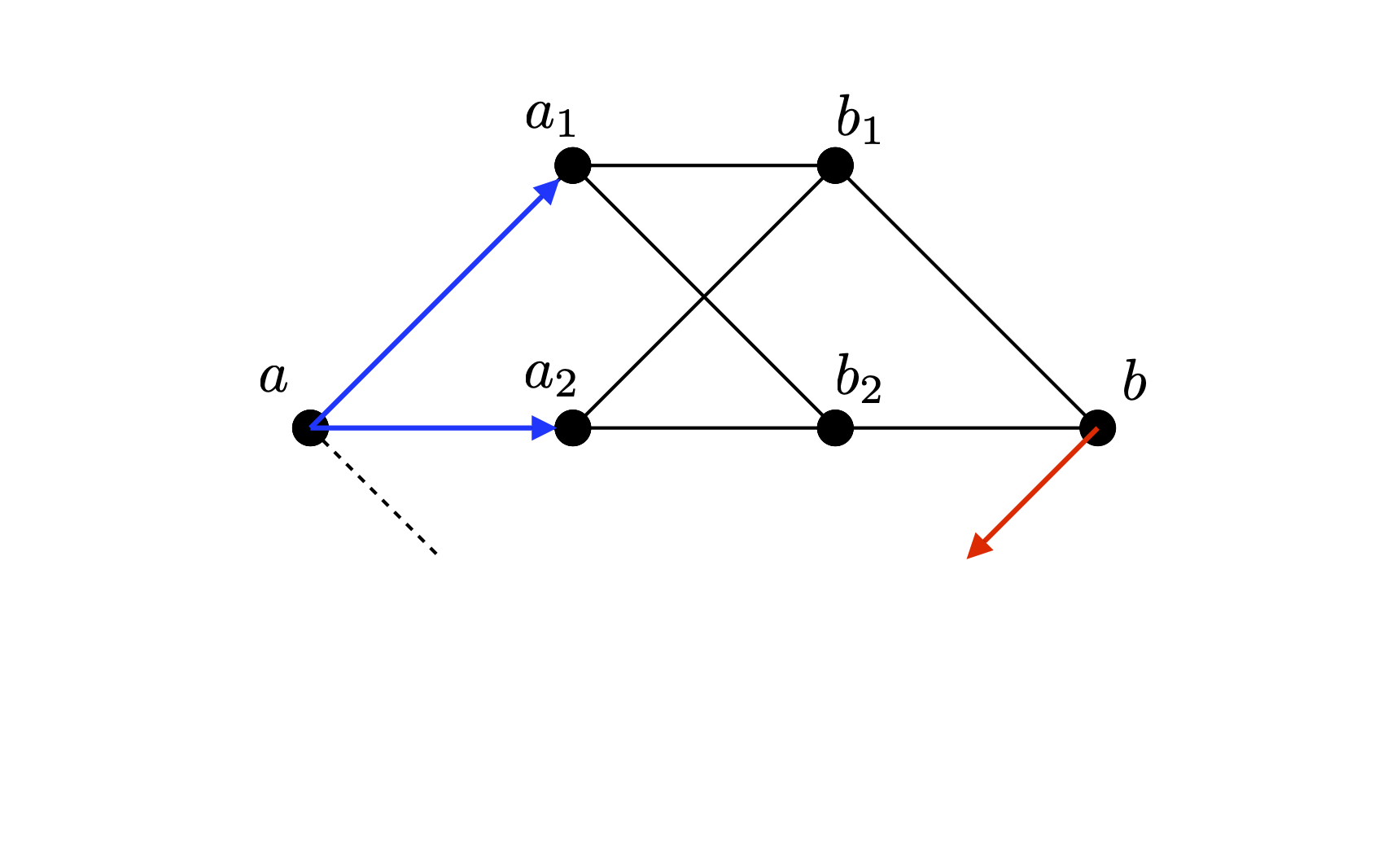}
\label{scenario_6_2}
}
\hfill
\subfloat[Coupling 3]{
\includegraphics[clip,trim=1in 2in 1in 0.7in,width=0.47\textwidth]{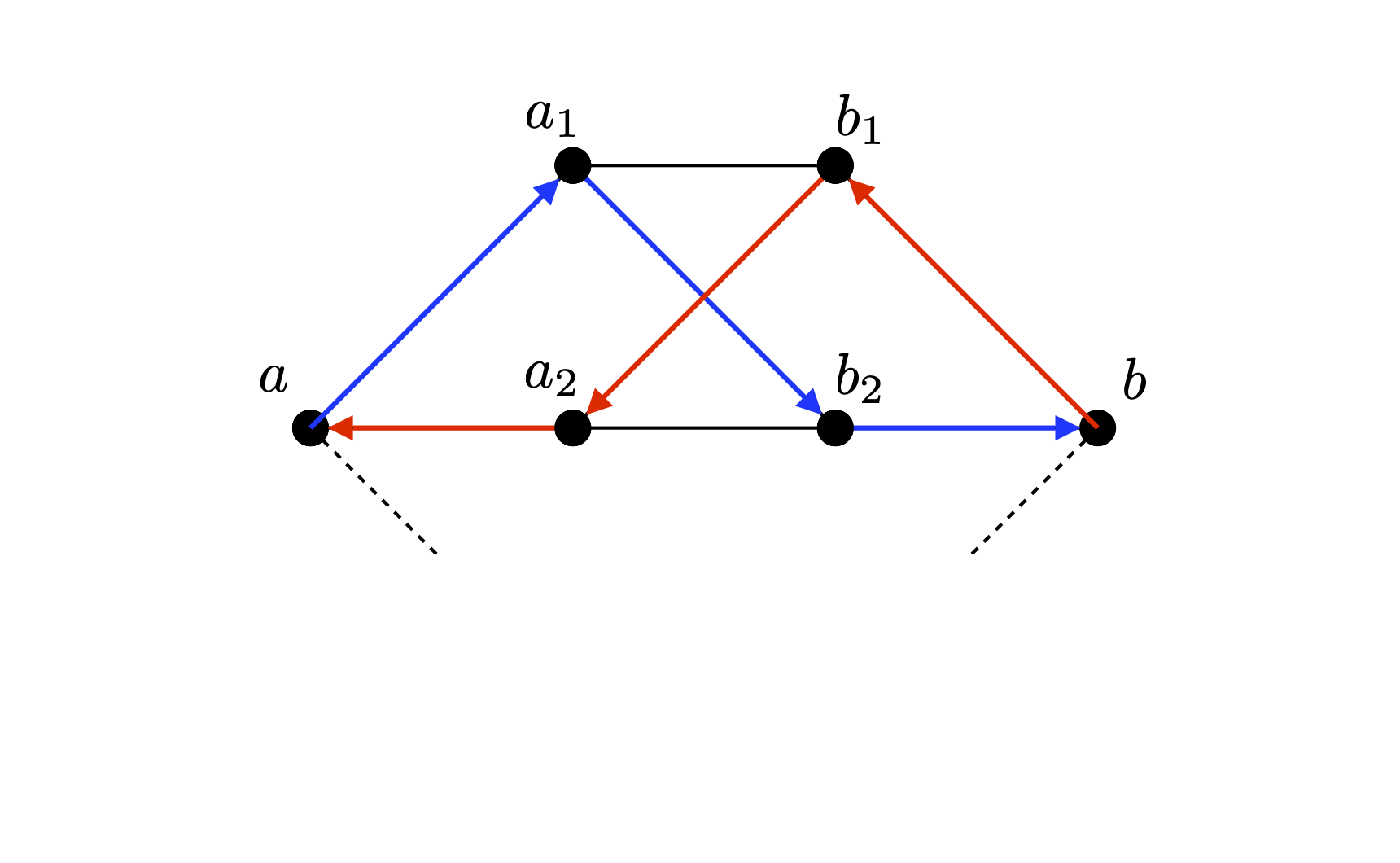}
\label{scenario_6_3}
}
\caption{Shown here is \hyperref[case_6]{Scenario 6}, the local configuration in $3$-regular graphs for which a variable-step coupling is required.
Each of three possible strategies is chosen with probability $\frac{1}{3}$.
In (b), Alice moves away from the local copy of $K_{2,2}$, while Bob moves onto it.
In (c), these actions are exchanged.
Finally, in (d), both walkers enter $K_{2,2}$ and coordinate their movements so as both to avoid collision and to exit $K_{2,2}$ in the same number of steps.
The example shown in the diagram consists of three steps per walker.
}
\label{variable_step}
\end{figure}
The coupling is as follows:
\begin{itemize}
\item (Figure~\ref{scenario_6_1}) With probability $\frac{1}{3}$, Alice moves to $c_a$, after which Bob moves to one of $b_1,b_2$ with equal chance.
\item (Figure~\ref{scenario_6_2}) With probability $\frac{1}{3}$, Bob moves to $c_b$ after Alice has moved to one of $a_1,a_2$ with equal chance.
\item (Figure~\ref{scenario_6_3}) With probability $\frac{1}{3}$, Alice moves to one of $a_1,a_2$ with equal chance and then performs simple random walk until hitting $\{a,b\}$.
This trajectory is of the form
\eq{
a \to a_{i_1} \to b_{j_2} \to a_{i_3} \to \cdots \to a_{i_{T-1}} \to a \qquad \text{($T$ even)}
}
or
\eq{
a \to a_{i_1} \to b_{j_2} \to a_{i_3} \to \cdots \to b_{j_{T-1}} \to b. \qquad \text{($T$ odd)}
}
Correspondingly, Bob follows the trajectory
\eq{
b \to b_{j_1} \to a_{i_2} \to b_{j_3} \to \cdots \to b_{j_{T-1}} \to b \qquad \text{($T$ even)}
}
or
\eq{
b \to b_{j_1} \to a_{i_2} \to b_{j_3} \to \cdots \to a_{i_{T-1}} \to a, \qquad \text{($T$ odd)}
}
where 
\eq{
b_{j_{2q-1}} &= \begin{cases}
b_1 &\text{if $2q<T$ and $b_{j_{2q}}=b_2$},\\
b_2 &\text{if $2q<T$ and $b_{j_{2q}}=b_1$},\\
b_1 &\text{with probability $\frac{1}{2}$, if $2q = T$}, \\
b_2 &\text{with probability $\frac{1}{2}$, if $2q = T$}, \\
\end{cases} \\
a_{i_{2q}} &= \begin{cases}
a_1 &\text{if $2q+1<T$ and $a_{i_{2q+1}}=a_2$},\\
a_2 &\text{if $2q+1<T$ and $a_{i_{2q+1}}=a_1$},\\
a_1 &\text{with probability $\frac{1}{2}$, if $2q+1 = T$}, \\
a_2 &\text{with probability $\frac{1}{2}$, if $2q+1 = T$}. \\
\end{cases}
}
\end{itemize}
In the first two possibilities, $T=1$.
In the third, $T$ is equal to the number of steps required for Alice to return to $\{a,b\}$.
Viewed together, the three possibilities lead to Alice and Bob choosing their next position uniformly from $N(a)$ and $N(b)$, respectively.
That is, \eqref{srw_condition} holds when $s = 0$.
Furthermore, in the third possible outcome, Alice's movements after transitioning to $\{a_1,a_2\}$ are independent of her history and follow the law of simple random walk.
Since Bob acts in a symmetric fashion---exchanging any $b_1$ with $b_2$, $b_2$ with $b_1$, $a_1$ with $a_2$, and $a_2$ with $a_1$ so that condition \hyperref[condition_2]{(ii)} is met---the same is true of his trajectory.
Therefore, \eqref{srw_condition} holds also when $1\leq s\leq T-1$, and so condition \hyperref[condition_1]{(i)} is satisfied.
Finally, case-by-case inspection reveals that condition \hyperref[condition_3]{(iii)} is also satisfied.

\section{Proof of Theorem~\hyperref[main_thm_a]{\ref*{main_thm}\ref*{main_thm_a}}: $d$-regular graphs, $d\geq4$} \label{d_regular}

\setcounter{subsection}{-1}
\subsection{Outline of coupling} \label{construction_outline}
In the setting of general $d\geq4$, the case-by-case approach used for $3$-regular graphs becomes intractable.
Nonetheless, parts of this section rely on $d\geq4$, and so the work of Section~\ref{3_regular} will not have been redundant.
Consider the following strategy:
\begin{itemize}
\item Suppose at time $t\geq0$, Alice is at vertex $A_t = a$ and is next to move.
\item Bob is at vertex $B_t = b$ which is not equal to $a$.
\item Assume a uniformly random $E_{t+1}=e\in N(a)$, which we call the `excluded vertex', has been selected such that (i) given $A_{t}=a$, the value of $E_{t+1}$ is conditionally independent of Alice's history; and (ii) if $b \in N(a)$, then $e = b$.
\item As depicted in Figure~\ref{algorithm}, we will specify a coupling (depending on $a$, $b$, and $e$) by which:
\begin{enumerate}[label=\textup{(\Alph*)}]
\item \label{alice_choices}
Alice will take two additional steps: the first to a uniformly random $A_{t+1}=a'\in N(a)\setminus\{e\}$, and the second to a uniformly random $A_{t+2}=a''\in N(a')$.
Given $(A_t,E_{t+1})$, the two steps are independent of 
Alice's history. 
Note that condition (ii) on $e$ guarantees $a'\neq b$.
\item \label{bob_choices} 
Bob will take one additional step to a uniformly random $B_{t+1} = b' \in N(b)$ which, given $B_t =b$, is independent of his history.
He will then select a uniformly random $E_{t+2} = e' \in N(b')$ that is (i) independent of Bob's history given his current position $B_{t+1}=b'$; and (ii) equal to $a''$ if $a'' \in N(b')$.
\end{enumerate}
\item Following this procedure (Alice moves $a\to a'$, Bob moves $b\to b'$, Alice moves $a'\to a''$, Bob selects $e'$), the roles of Alice and Bob will have been exchanged, and so the procedure---which is made precise in Section~\ref{formal_construction}---can be repeated.
\end{itemize}

\begin{figure}
\subfloat[Alice's first step]{
\includegraphics[clip,trim=4.4in 2in 1.8in 1in,height=1.5in]{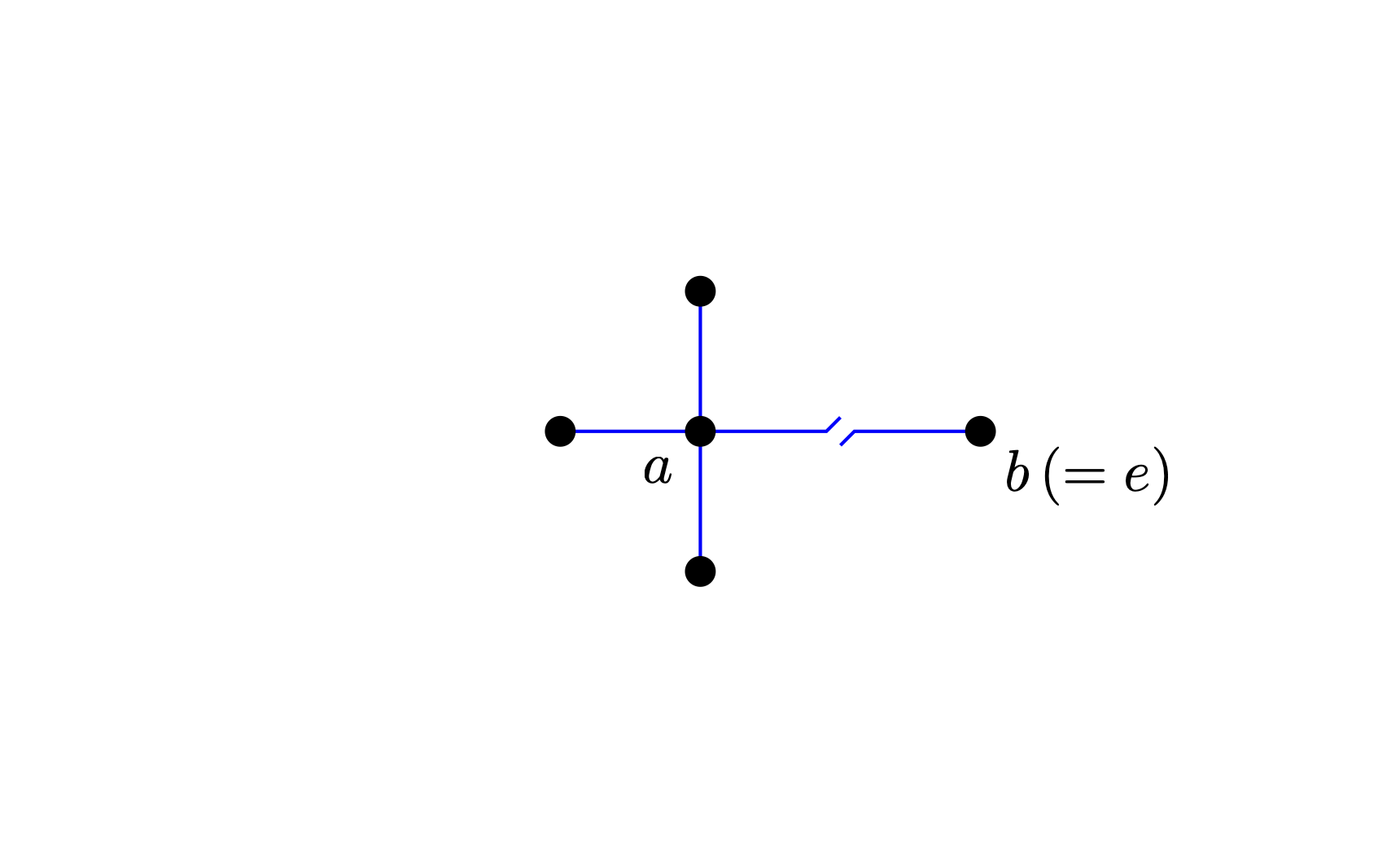}
\label{alice_1}
}
\hspace{0.6in}
\subfloat[Alice's second step]{
\includegraphics[clip,trim=3in 2in 3in 1in,height=1.5in]{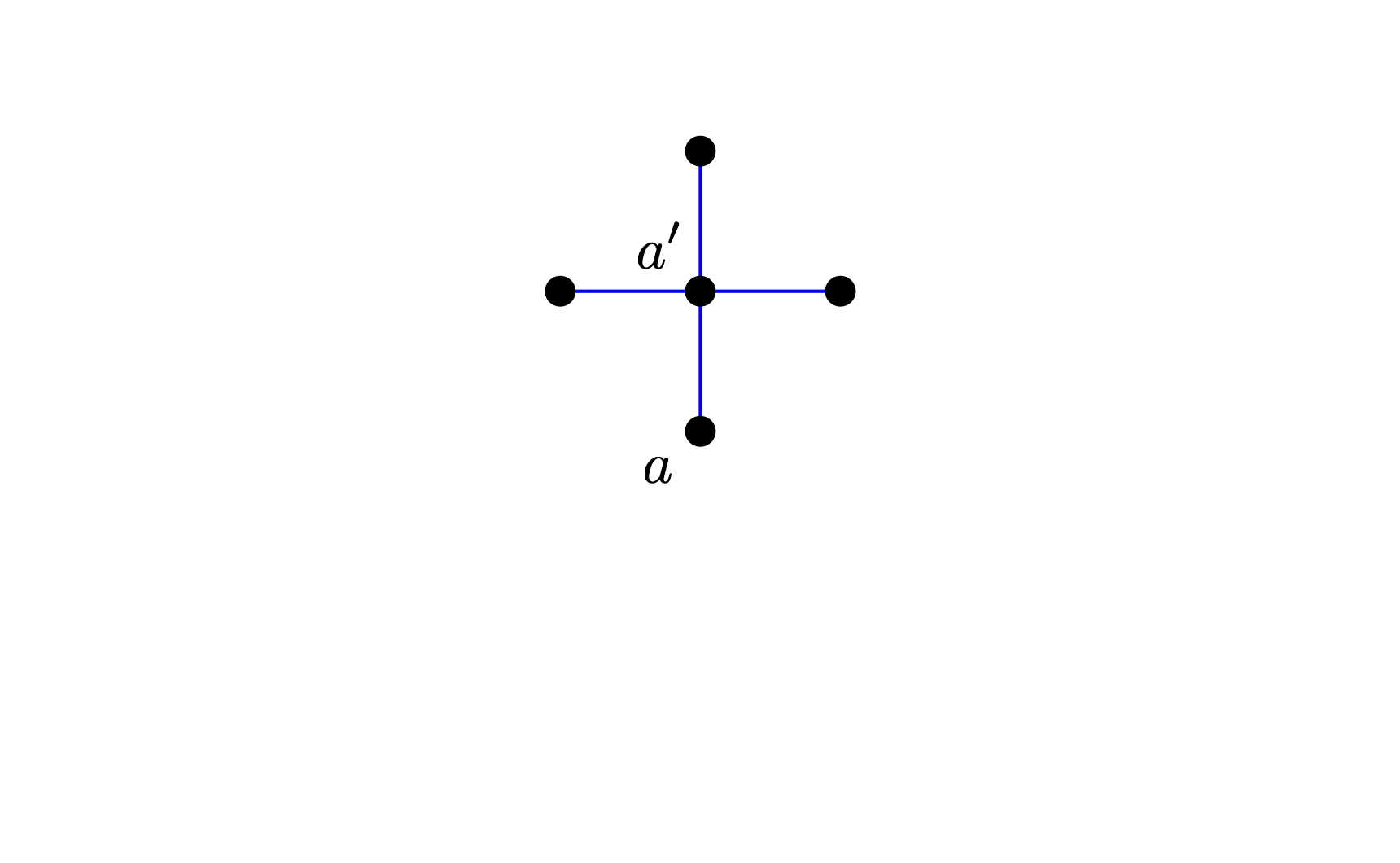}
\label{alice_2}
}
\\
\subfloat[Bob's step]{
\includegraphics[clip,trim=4.4in 2.8in 1.8in 1in,height=1.21in]{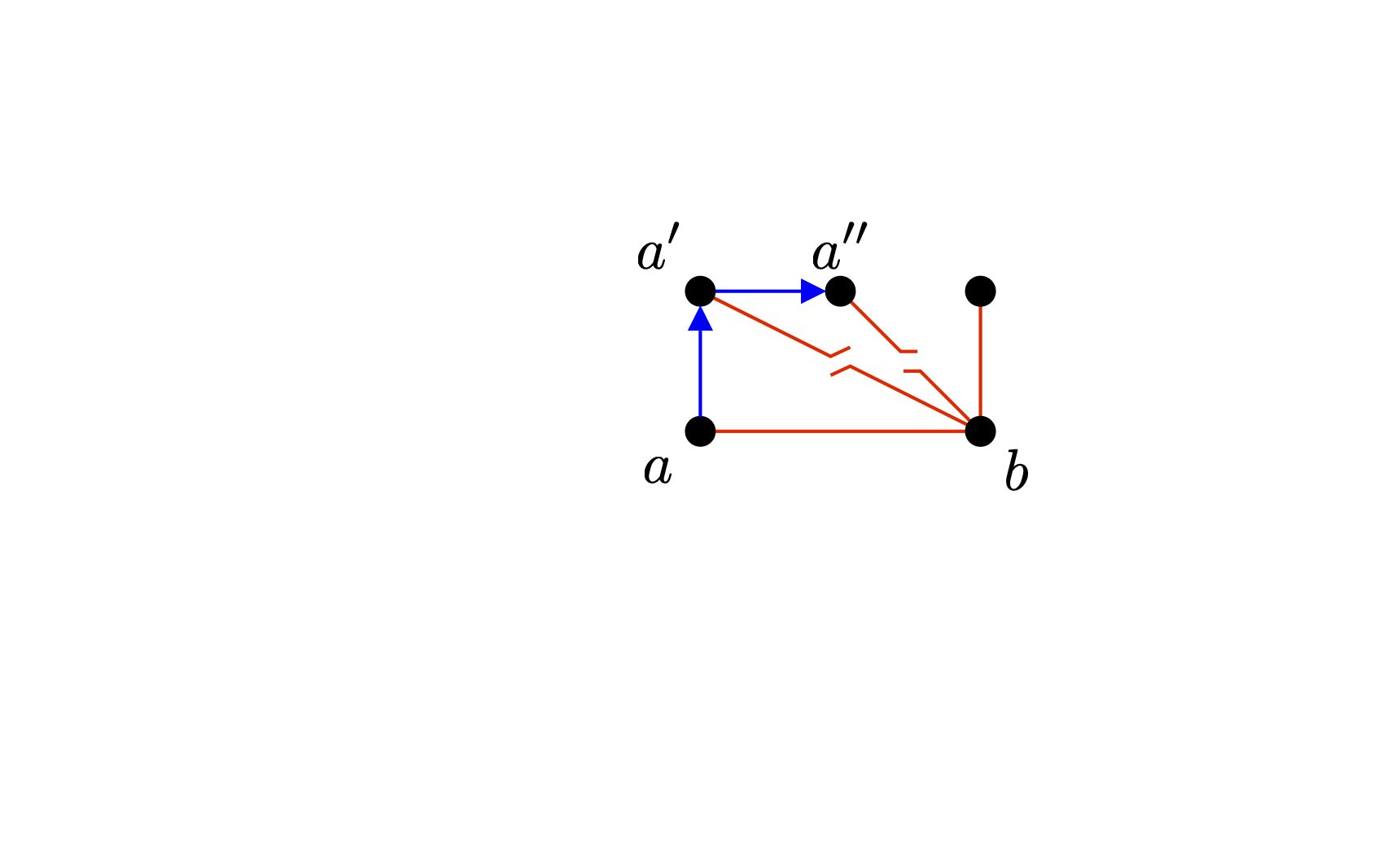}
\label{bob_1}
}
\hspace{0.6in}
\subfloat[Bob's next step]{
\includegraphics[clip,trim=5.3in 2.8in 0.7in 1in,height=1.21in]{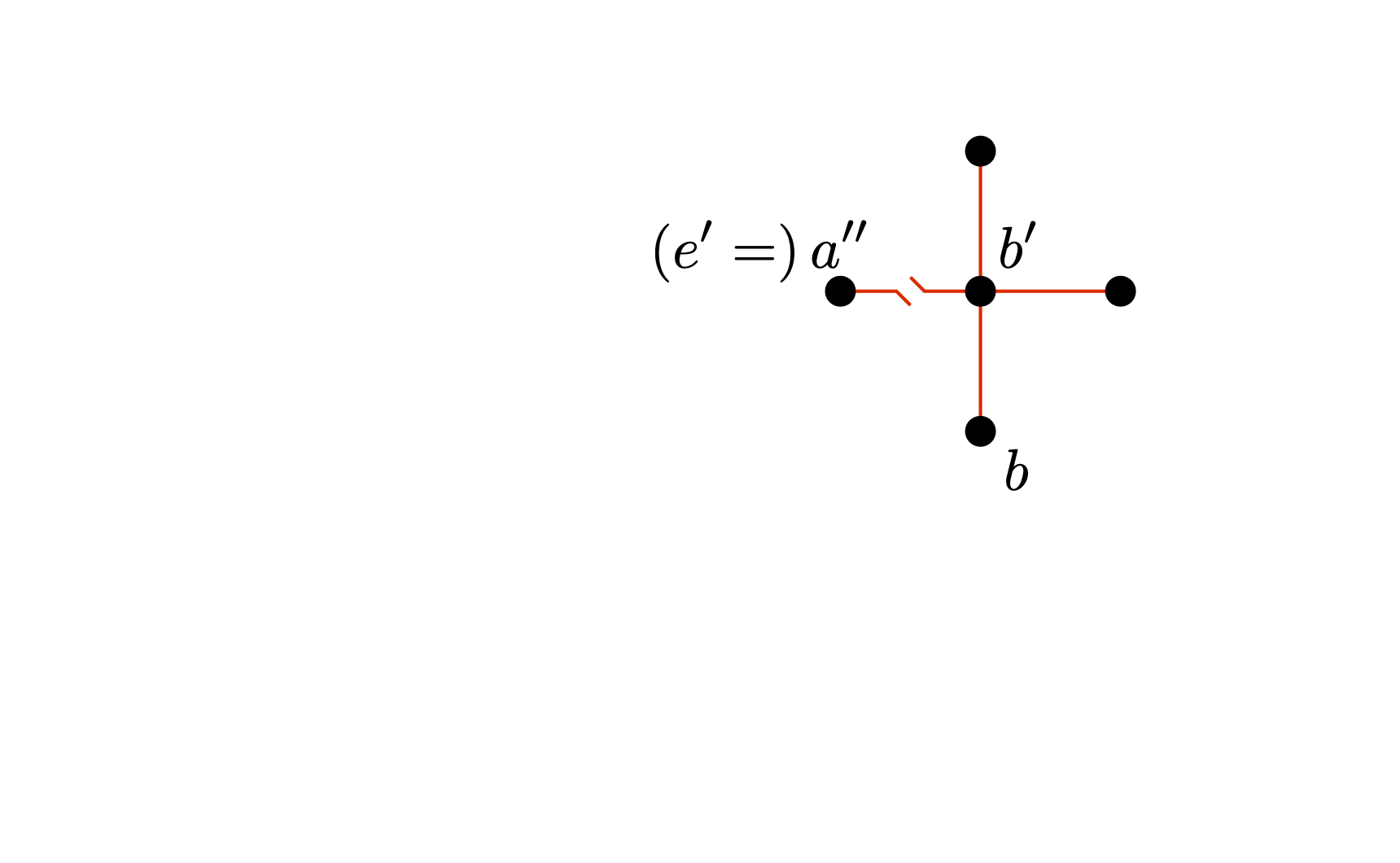}
\label{bob_2}
}
\caption{A single iteration of the avoidance algorithm for $d$-regular graphs, shown here with $d=4$.
In (a), Alice is at vertex $a$, Bob at $b$, and Alice is next to move.  She may choose any neighboring vertex except $e$, where $e$ is assumed equal to Bob's position if $b\in N(a)$.
In (b), Alice has chosen to move from $a$ to $a'$.
Her next step can be to any neighbor of $a'$.
In (c), Alice has chosen to move from $a'$ to $a''$.
Once Bob knows her trajectory $a\to a'\to a''$, he can move to any neighboring vertex not equal to $a'$ or $a''$.
In (d), Bob has chosen to move from $b$ to $b'$.
He then selects a neighbor $e' \in N(b')$ to be his excluded vertex for the next iteration of the algorithm, in which he and Alice exchange roles.
If $a''\in N(b')$, then Bob must take $e'=a''$.
}
\label{algorithm}
\end{figure}

The subtleties of the construction arise from the need to guarantee $b' \notin \{a',a''\}$ while also preserving the uniformity of Alice's and Bob's choices.
To this end, we will establish some graph theoretic and combinatorial properties in Sections~\ref{subgraph_restatements} and~\ref{pair_neighbors}.
It should be otherwise intuitively clear that if Alice and Bob adhere to the stipulated protocol, they will each carry the law of simple random walk.
We formally check this fact in Section~\ref{srw_verification}.

\subsection{Step 1: graph theoretic preliminaries} \label{subgraph_restatements}

Recall the graph $H_d$ from Figure~\ref{bad_d}.

\begin{lemma} \label{bad_subgraph_restated}
If $G=(V,E)$ is a $d$-regular graph, then the following three statements are equivalent:
\begin{enumerate}[label=\textup{(\arabic*)}]
\item $G$ contains no copy of $H_d$ as a subgraph;
\item for every $a,b\in V$ with $a\neq b$, we have
$N(a)\cup\{a\}\neq N(b)\cup\{b\}$; and
\item for every $a,b\in V$ with $N(a)\neq N(b)$, the set $N(a)\setminus(N(b)\cup\{b\})$ is nonempty.
\end{enumerate}
\end{lemma}

\begin{proof}
First we argue that (1) and (2) are equivalent.
Suppose $G$ contains a copy of $H_d$.
Simple inspection of Figure~\ref{bad_d} reveals that the two distal vertices---call them $a$ and $b$---satisfy $N(a)\cup\{a\} = N(b)\cup\{b\}$.
Conversely, suppose distinct vertices $a,b\in V$ are such that $N(a)\cup\{a\} = N(b)\cup\{b\}$.
Because $a\neq b$, the assumption $b\in N(a)\cup\{a\}$ implies $b\in N(a)$.
Furthermore, we have 
\eq{
N(a)\setminus\{b\}=(N(a)\cup\{a\})\setminus\{a,b\}=(N(b)\cup\{b\})\setminus\{a,b\}=N(b)\setminus\{a\},
}
meaning $a$ and $b$ share their remaining $d-1$  neighbors.
Therefore, the subgraph of $G$ induced by $N(a)\cup\{a\}$ contains a copy of $H_d$.

Next we show that (2) and (3) are equivalent.
Assuming (2), let us consider $a,b$ such that $N(a)\neq N(b)$.
Suppose toward a contradiction that $N(a)\setminus(N(b)\cup\{b\})$ is empty. 
That is, $N(a) \subset N(b)\cup\{b\}$.
Because $N(a)\neq N(b)$ but $|N(a)|=|N(b)|=d$, we also have $N(a)\setminus N(b)\neq\varnothing$. 
It now follows that $b\in N(a)$, or equivalently $a\in N(b)$.
Hence $N(a)\cup\{a\}\subset N(b)\cup\{b\}$.
Since $|N(a)\cup\{a\}| = |N(b)\cup\{b\}|$, this containment is actually an equality, thereby contradicting (2).

Finally, let us assume (3) and suppose toward a contradiction that $N(a)\cup\{a\}=N(b)\cup\{b\}$ for some $a\neq b$.
In particular, we have $a \in N(b)\setminus N(a)$, implying $N(a)\neq N(b)$.
But now (3) guarantees $N(a)\setminus(N(b)\cup\{b\})$ is nonempty, a clear contradiction to our supposition.
\end{proof}

\subsection{Step 2: combinatorics of compatible moves} \label{pair_neighbors}
In this section and Section~\ref{formal_construction}, we temporarily fix a triple $(a,b,e)$ such that $b\neq a$,  $e\in N(a)$, and $e = b$ if $b\in N(a)$.
Consider the sets
\eq{
\AA &\coloneqq \{(a',a'') : a'\in N(a)\setminus\{e\},a''\in N(a')\}, \qquad
\BB \coloneqq \{(b',e') : b'\in N(b),e'\in N(b')\}.
}
In other words, $\AA$ and $\BB$ encode the possible pairs of choices from Alice and Bob, respectively, described in \ref{alice_choices} and \ref{bob_choices} of Section~\ref{construction_outline}.

\begin{defn} \label{compatible_def}
We say that $(a',a'') \in \AA$ and $(b',e')\in \BB$ are \textit{compatible}, and write $(a',a'')  \vdash(b',e')$, when the following two conditions hold:
\begin{enumerate}[label=\textup{(\roman*)}]
\item $b'\notin\{a',a''\}$; and 
\item if $a''\in N(b')$, then $e' = a''$.
\end{enumerate}
If either of these two conditions fail, we will write $(a',a'')\nvdash(b',e')$.
\end{defn}

For any $\AA_0 \subset \AA$, define the set
\eq{
\cmp(\AA_0) \coloneqq \{(b',e')\in\BB : (a',a'')  \vdash(b',e') \text{ for some $(a',a'')\in\AA_0$}\}.
}
The following result will be essential in allowing us to couple Alice's choice from $\AA$ with Bob's choice from $\BB$.

\begin{lemma} \label{hall_prep_1}
For any $\AA_0\subset\AA$,
\eq{
 \frac{d}{d-1}|\AA_0| \leq |\cmp(\AA_0)|.
}
\end{lemma}

Before proving the lemma, we make a simplifying claim.

\begin{claim} \label{hall_prep_claim}
Suppose $(a'_1,a'') \in \AA_0\subset\AA$ and $(a'_2,a'')\in\AA$.
If $\AA_1 = \AA_0 \cup \{(a'_2,a'')\}$, then 
\eeq{ \label{claimed_implication}
 \frac{d}{d-1}|\AA_1| \leq |\cmp(\AA_1)| \quad \implies \quad  \frac{d}{d-1}|\AA_0| \leq |\cmp(\AA_0)|.
}
\end{claim}

\begin{proof}
If $(a_2',a'') \in \AA_0$, then $\AA_1=\AA_0$ and the claim is trivial.
Therefore, let us assume henceforth $(a_2',a'') \notin \AA_0$ so that $|\AA_1| = |\AA_0|+1$.
Suppose $\cmp(\AA_1)\setminus\cmp(\AA_0)$ contains some $(b',e')$.
In particular, $(a_1',a'') \nvdash(b',e')$ but $(a_2',a'')\vdash(b',e')$.
Observe that
\eq{ 
(a_2',a'')\vdash(b',e') \quad \implies \quad \text{$b'\notin\{a_2',a''\}$ and $a'' \notin N(b')\setminus\{e'\}$}.
}
Meanwhile,
\eq{ 
(a_1',a'')\nvdash(b',e') \quad \implies \quad \text{$b' \in \{a_1',a''\}$ or $a'' \in N(b')\setminus\{e'\}$}.
}
Viewed together, these two implications allow just one possibility: $b' = a_1'$, and therefore $a''\in N(a_1') = N(b')$ so that $a''$ must be $e'$.
We have thus argued that $\cmp(\AA_1)\setminus\cmp(\AA_0)$ contains at most one element.
Consequently, if
\eq{
1\leq \frac{d}{d-1} \leq \frac{|\cmp(\AA_1)|}{|\AA_1|},
}
then
\eq{
\frac{|\cmp(\AA_1)|}{|\AA_1|} \leq \frac{|\cmp(\AA_1)|-1}{|\AA_1|-1} = \frac{|\cmp(\AA_1)|-1}{|\AA_0|} \leq \frac{|\cmp(\AA_0)|}{|\AA_0|}.
}
The claimed implication \eqref{claimed_implication} is now evident.
\end{proof}

\begin{proof}[Proof of Lemma~\ref{hall_prep_1}]
Given $\AA_0\subset\AA$, we begin by defining the set
\eq{
\AA_2 \coloneqq \{a''\in V : (a',a'')\in \AA_0 \text{ for some $a'\in N(a)$}\}.
}
For the purpose of proving the lemma, we may assume by Claim~\ref{hall_prep_claim} that
\eeq{ \label{inclusion_assumption}
a'' \in \AA_2, (a',a'')\in\AA \quad \implies \quad (a',a'')\in\AA_0.
}

{\bf Case 1: $|\AA_2| \geq d+2$.}
Consider any $b' \in N(b)$.
Because $|N(b')| = d$, the set $\AA_2\setminus(N(b')\cup\{b'\})$ contains some $a''$. 
By definition of $\AA_2$, there is some $a'\in N(a)$ for which $(a',a'')\in\AA_0$.
Because $a''\in N(a')\setminus (N(b')\cup\{b'\})$, we must have $a'\neq b'$ and $a''\neq b'$.
Consequently, $(a',a'')\vdash(b',e')$ for every $e'\in N(b')$.
We have thus argued that $\cmp(\AA_0) = \BB$, and so the claim holds trivially:
\eeq{ \label{trivial_whole}
\frac{d}{d-1}|\AA_0| \leq \frac{d}{d-1}|\AA| = d^2 = |\BB|.
}

{\bf Case 2: $|\AA_2| = d+1$.}
If $\cmp(\AA_0)$ is all of $\BB$, then we are done by \eqref{trivial_whole}.
Otherwise, there is some $(b_1,e_1) \in \BB\setminus\cmp(\AA_0)$, meaning the following implication is true:
\eeq{ \label{not_compatible_consequence}
(a',a'') \in \AA_0 \quad &\implies \quad b_1\in\{a',a''\} \text{ or } a''\in N(b_1)\setminus\{e_1\} \\
&\implies \quad a'' \in \{b_1\} \cup N(b_1).
}
This shows that $\AA_2 \subset \{b_1\} \cup N(b_1)$, but since $|\AA_2| = d+1$, we must actually have 
\eeq{ \label{A2_equality}
\AA_2 = \{b_1\} \cup N(b_1).
}
In particular, there is some $a_1\in N(a)$ so that $(a_1,e_1)\in \AA_0$.
In light of \eqref{not_compatible_consequence}, though, we can only have $a_1= b_1$.
(In particular, $b_1 \in N(a)$ and $b_1\neq e$.)
It thus follows from \eqref{inclusion_assumption} that 
\eeq{ \label{not_in_neighborhood}
e_1 \notin N(a') \quad \text{for any $a'\in N(a)\setminus\{b_1,e\}$},
}
since otherwise $\cmp(\AA_0)$ would contain $(b_1,e_1)$.
We claim that, as a consequence,
\eeq{ \label{lost_one}
|N(a')\cap\AA_2| \leq d-1 \quad \text{for any $a'\in N(a)\setminus\{b_1,e\}$}.
}
Indeed, because $\AA_2 = N(b_1)\cup\{b_1\}$, we have
\eq{
N(a')\cap \AA_2 \subset (N(b_1)\setminus\{e_1\})\cup\{b_1\}.
}
The right-hand side above has cardinality $d$.
So if $|N(a')\cap \AA_2|$ were at least $d$, then the above containment would be an equality:
\eeq{ \label{bad_equality}
N(a')=N(a')\cap\AA_2=(N(b_1)\setminus\{e_1\})\cup\{b_1\}.
}
This would in turn imply $b_1\in N(a')$ and hence $a'\in N(b_1)$.
But clearly $a'\notin N(a')$, and so we are left to conclude from \eqref{bad_equality} that $a'=e_1$.
This yields a contradiction to Lemma~\ref{bad_subgraph_restated}, since now \eqref{bad_equality} shows
\eq{
N(a')\cup\{a'\} = N(a')\cup\{e_1\} = N(b_1) \cup \{b_1\}.
}
To avoid this contradiction, \eqref{lost_one} must hold, which means
\eeq{ \label{A_bound_1}
(b_1,e_1)\in \BB\setminus\cmp(\AA_0) \quad \implies \quad |\AA_0| &\stackref{inclusion_assumption}{=} \sum_{a'\in N(a)\setminus\{e\}} |N(a')\cap\AA_2| \\
&\stackref{A2_equality}{=} |N(b_1)| + \sum_{a'\in N(a)\setminus\{b_1,e\}} |N(a')\cap\AA_2| \\
&\stackrefp{A2_equality}{\leq}  d + (d-2)(d-1).
}

More generally, suppose that $(b_1,e_1),(b_1,e_2),\cdots,(b_1,e_\ell) \in \BB\setminus\cmp(\AA_0)$, where $\ell\geq2$ and $e_1,\dots,e_\ell$ are all distinct.
By the same argument as the one leading to \eqref{not_in_neighborhood}, we have $e_1,\dots,e_\ell\notin N(a')$ for every $a' \in N(a)\setminus\{b_1,e\}$.
Since $\AA_2 = N(b_1)\cup\{b_1\}$, this observation shows
\begin{align}
N(a')\cap \AA_2 &\subset (N(b_1)\setminus\{e_1,\dots,e_\ell\})\cup\{b_1\} \notag \\
\quad \implies \quad |N(a')\cap \AA_2| &\leq |(N(b_1)\setminus\{e_1,\dots,e_\ell\})\cup\{b_1\}
| = d-\ell+1. \label{lost_at_least_one}
\end{align}
The resulting bound on $|\AA_0|$ is now
\eq{ 
|\AA_0| 
&= |N(b_1)| + \sum_{a'\in N(a)\setminus\{b_1,e\}} |N(a')\cap\AA_2|
\leq  d + (d-2)(d-\ell+1).
}
Note that $\ell=2$ yields the same bound as \eqref{A_bound_1}, and so we can write the single statement
\eeq{ \label{A_bound_2}
(b_1,e_1),\dots,(b_1,e_\ell) \in \BB\setminus\cmp(\AA_0) \quad \implies \quad |\AA_0| \leq d + (d-2)(d-(\ell \vee 2)+1).
}

We now separately compute $|\cmp(\AA_0)|$.
Let $\ell$ be the maximum integer such that there are distinct $e_1,\dots,e_\ell\in N(b_1)$ for which
\eeq{ \label{not_in_cmp}
(b_1,e_1),\dots,(b_1,e_\ell) \in \BB\setminus\cmp(\AA_0). 
}
Consider any $b'\in N(b)\setminus\{b_1\}$ and any $e'\in N(b')$.
Because of our earlier deduction in \eqref{A2_equality} that $\AA_2 = N(b_1)\cup\{b_1\}$, where $b_1\in N(a)\setminus\{e\}$, the assumption \eqref{inclusion_assumption} forces $(b_1,a'')\in\AA_0$ for every $a''\in N(b_1)$.
In particular, if $e'\in N(b_1)$, then $(b_1,e')\in\AA_0$, and so
\eq{
b' \notin \{b_1,e'\} \quad \implies \quad
(b_1,e')\vdash(b',e') \quad \implies \quad (b',e') \in \cmp(\AA_0).
}
If instead $e'\notin N(b_1)$, then $N(b_1)\neq N(b')$.
In this case, Lemma~\ref{bad_subgraph_restated} tells us that 
$N(b_1)\setminus N(b')$ contains some $a''\neq b'$, and so
\eq{
b'\notin\{b_1,a''\}, a''\notin N(b') \quad \implies \quad 
(b_1,a'')\vdash (b',e') \quad \implies \quad (b',e') \in \cmp(\AA_0).
}
We have thus shown that $\cmp(\AA_0)$ contains every $(b',e')\in\BB$ for which $b'\neq b_1$.
By our choice of $\ell$, we now have
\eq{
\cmp(\AA_0) = \BB\setminus\{(b_1,e_1),\dots,(b_1,e_\ell)\} \quad \implies \quad |\cmp(\AA_0)| = d^2-\ell \geq d^2-(\ell\vee2).
}
Pairing this lower bound for $|\cmp(\AA_0)|$ with the upper bound \eqref{A_bound_2} for $|\AA_0|$---and assuming $\ell\geq2$ in order to simplify notation---we have
\eeq{ \label{raw_bound}
\frac{|\cmp(\AA_0)|}{|\AA_0|} \geq \frac{d^2-\ell}{d+(d-2)(d-\ell+1)} = \frac{d^2-\ell}{d^2-\ell d + 2\ell - 2}.
}
Observe that
\eq{
\ell\geq 2 \quad \implies \quad \frac{3\ell-2}{\ell-1} \leq 4 \leq d,
}
and so
\eeq{ \label{ingredient}
(d^2-\ell d + 2\ell - 2)d &= d^3 - \ell d^2 + 2\ell d - 2d \\
&= d^3 - d^2 - \ell d - (\ell-1)d^2 + (3\ell-2)d \\
&\leq d^3 - d^2 - \ell d 
\leq d^3 - d^2 - \ell d + \ell = (d^2-\ell)(d-1).
}
Together, \eqref{raw_bound} and \eqref{ingredient} produce the desired inequality:
\eq{
\frac{|\cmp(\AA_0)|}{|\AA_0|} \geq \frac{d}{d-1}.
}

{\bf Case 3: $|\AA_2| \leq d$.}
We will handle this final case by eventually splitting our argument along three sub-cases, to be specified later.
To begin, we enumerate the elements of $N(a)$ as $a_1,\dots,a_{d-1},a_d=e$, and then correspondingly label the elements of $N(b)$ as $b_1,\dots,b_d$ in such a way that
\begin{subequations}
\label{label_conditions}
\begin{align} \label{label_1}
b_j=a_i \in N(a) \quad \implies \quad j=i.
\end{align}
Next we enumerate the elements of $\AA_2$ as $a^1,\dots,a^K$, where $K \leq d$, such that
\begin{align} \label{label_2}
a^k = b_j \in N(b) \quad \implies \quad k=j.
\end{align}
Finally, we enumerate the elements of each $N(b_j)$ as $b_j^1,\dots,b_j^d$, such that
\begin{align} \label{label_3}
b_j^\ell = a^k \in \AA_2 \quad \implies \quad \ell=k.
\end{align}
\end{subequations}
We will now use these enumeration schemes to directly construct a subset of $\cmp(\AA_0)$ large enough to satisfy the lemma's claim.

In light of \eqref{inclusion_assumption}, the set $\AA_0$ can be expressed as the disjoint union
\eeq{ \label{disjoint_rep}
\AA_0 = \biguplus_{k=1}^{d} \AA_0^k, 
}
where
\eq{
\AA_0^k \coloneqq 
\begin{cases} \{(a_i,a^k) : a_i \in N(a)\cap N(a^k), 1 \leq i \leq d-1\} &\text{if }1\leq k\leq K, \\
\varnothing &\text{if }K<k\leq d.
\end{cases}
}
(Without \eqref{inclusion_assumption}, $\AA_0$ would only be a subset of the union in \eqref{disjoint_rep}.)
Note that $|\AA_0^k|\leq d-1$ for each $k$, which implies
\eeq{ \label{proportions}
|\AA_0^k| + 1 \geq |\AA_0^k| + \frac{1}{d-1}|\AA_0^k| = \frac{d}{d-1}|\AA_0^k|.
}
For each $(a_i,a^k) \in \AA_0^k$, we claim that 
\eeq{
(a_i,a^k) \vdash (b_j,b_j^k) \quad \text{for all $j \neq i,k$}. \label{avoid_i_k}
}
Indeed, when, \eqref{label_1} guarantees $b_j \neq a_i$ when $j\neq i$, \eqref{label_2} guarantees $b_j \neq a^k$ when $j\neq k$, and \eqref{label_3} ensures $b_j^k = a^k$ if $a^k \in N(b_j)$.

Our goal is to construct subsets $\BB^1,\dots,\BB^{d} \subset \BB$ such that
\eeq{ \label{B_condition}
\BB^k \subset \cmp(\AA_0) \cap \{(b_j,b_j^k) : 1 \leq j \leq d\} \quad \text{for each $k=1,\dots,d$}.
}
Such subsets are automatically pairwise disjoint, since
\eq{
(b_j,b_j^k) = (b_{j'},b_{j'}^{\ell}) \quad \implies \quad
 \begin{cases}
 b_j = b_{j'}\ \implies \ j = j' \\
 b_j^k = b_{j'}^\ell
\end{cases}
 \implies \quad b_j^k = b_j^\ell \quad \implies \quad k = \ell.
}
Therefore, we will ultimately have
\eeq{ \label{disjoint_counting}
|\cmp(\AA_0)| \geq \Big|\bigcup_{k=1}^d \BB^k\Big|
= \sum_{k=1}^d |\BB^k|.
}
Furthermore, the sets $\BB^1,\dots,\BB^d$ we identify will have the property that
\eeq{ \label{sum_condition}
 \sum_{k=1}^d |\BB^k| \geq K + \sum_{k=1}^K |\AA_0^k|, 
}
from which the lemma's claim follows:
\eq{
|\cmp(\AA_0)| \stackref{disjoint_counting}{\geq} \sum_{k=1}^d |\BB^k|
\stackref{sum_condition}{\geq} \sum_{k=1}^K (|\AA_0^k|+1) 
\stackref{proportions}{\geq} \frac{d}{d-1}\sum_{k=1}^K |\AA_0^k|
\stackref{disjoint_rep}{=} \frac{d}{d-1}|\AA_0|.
}
We are thus left only with the task of exhibiting $\BB^1,\dots,\BB^d$ satisfying \eqref{B_condition} and \eqref{sum_condition}.
Our definition of $\BB^k$ will depend on the cardinality of $\AA_0^k$. \\

\noindent \underline{$|\AA_0^k| = 0$}:
Here $K<k\leq d$, and we simply set
\eeq{
\label{B_0}
\BB^k=\varnothing.
}

\noindent \underline{$|\AA_0^k| = 1$}:
Say $\AA_0^k = \{(a_i,a^k)\}$; then set 
\eeq{ \label{B_1}
\BB^k = \{(b_j,b_j^k) : 1\leq j\leq d, j\neq i,k\}.
}
The observation \eqref{avoid_i_k} shows $\BB^k \subset \cmp(\AA_0^k)$, and we note
\eeq{ \label{ineq_1}
|\BB^k| = d-2 \geq 2 = |\AA_0^k| + 1.
}

\noindent \underline{$2\leq |\AA_0^k| \leq d-2$}: Say $\AA_0^k \supset \{(a_{i_1},a^k),(a_{i_2},a^k)\}$ with $i_1\neq i_2$.
As any $j$ is equal to at most one of $i_1$ and $i_2$, it follows from \eqref{avoid_i_k} that $(b_j,b_j^k) \in \cmp(\AA_0^k)$ for every $j\neq k$. 
So setting 
\eeq{ \label{B_2}
\BB^k = \{(b_j,b_j^k) : 1\leq j \leq d, j \neq k\}
} 
again results in  $\BB^k \subset \cmp(\AA_0^k)$, and now
\eeq{ \label{ineq_2}
|\BB^k| = d-1 \geq |\AA_0^k| + 1.
}

\noindent \underline{$|\AA_0^k| = d-1$}: 
Our definition of $\BB^k$ when $|\AA_0^k| = d-1$ will depend on which of the three sub-cases below we find ourselves in.
In any circumstance, however, the same logic as above yields
\eeq{ \label{more_than_1_consequence}
\{(b_j,b_j^k) : 1\leq j \leq d, j \neq k\} \subset \cmp(\AA_0^k).
}
If $a^k\notin N(b)$, then we have the additional guarantee that $a^k\neq b_k$.
So by taking any $i\neq k$, $1\leq i\leq d-1$, we see that
\eq{
|\AA_0^k| = d-1 \quad \implies \quad \AA_0^k = N(a)\setminus\{e\} \quad 
&\stackrel{\phantom{\mbox{\footnotesize\eqref{label_1},\eqref{label_3}}}}{\implies} \quad (a_i,a^k) \in \AA_0^k \\ &\stackrel{\mbox{\footnotesize\eqref{label_1},\eqref{label_3}}}{\implies} \quad (b_k,b_k^k) \in \cmp(\AA_0^k).
}
In this case, we can improve upon \eqref{more_than_1_consequence}:
\eeq{ \label{not_neighbor_consequence}
\{(b_j,b_j^k) : 1 \leq j\leq d\} \subset \cmp(\AA_0^k) \quad \text{whenever $|\AA_0^k| = d-1$, $a^k\notin N(b)$}.
}
On the other hand, if $a^k \in N(b)$, then we will appeal to the following claim:
\eeq{ \label{key_implication}
a^k \in N(b) \quad \implies \quad \text{there is some $j \neq k$ for which $a^k\notin N(b_{j})$}.
}
To verify this claim, let us suppose $a^k \in N(b)$.
Note that \eqref{label_2} forces $a^k = b_k$.
If the conclusion of \eqref{key_implication} were false, then $N(a^k) \supset \{b_j : j\neq k\}$, which in turn gives
\eq{
N(a^k) \cup\{a^k\} \supset \{b\} \cup \{b_j : j\neq k\} \cup \{b_k\} = N(b) \cup \{b\} \implies N(a^k)\cup\{a^k\} = N(b) \cup\{b\}.
}
As this possibility violates Lemma~\ref{bad_subgraph_restated}, we have proved \eqref{key_implication}. \\


{\it Case 3a: $|\AA_0^k|\leq d-2$ for all values of $k$.}
In this first possibility, $\BB^1,\dots,\BB^d$ are all defined via \eqref{B_0}, \eqref{B_1}, or \eqref{B_2}.
Consequently, \eqref{sum_condition} is immediate from \eqref{ineq_1} and \eqref{ineq_2}. \\

{\it Case 3b: $|\AA_0^k|=d-1$ for exactly one value of $k$.}
Suppose $|\AA_0^{k^*}| = d-1$ and $|\AA_0^k| \leq d-2$ for $k\neq k^*$.
For each $k\neq k^*$, the set $\BB^k$ is defined via \eqref{B_0}, \eqref{B_1}, or \eqref{B_2}.
If $a^{k^*} \notin N(b)$, then \eqref{not_neighbor_consequence} allows us to set 
\eq{
\BB^{k^*} = \{(b_j,b_j^{k^*}): 1 \leq j \leq d\},
} 
in which case
\eq{
|\BB^{k^*}| = d = |\AA_0^{k^*}|+1.
}
This relation, combined with \eqref{ineq_1} and \eqref{ineq_2}, leads to \eqref{sum_condition}.

On the other hand, even if $a^{k^*} \in N(b)$, \eqref{more_than_1_consequence} allows us to at least take
\eq{
\BB^{k^*} = \{(b_j,b_j^{k^*}): 1 \leq j \leq d, j\neq k^*\},
}
so that
\eeq{ \label{ineq_3}
|\BB^{k^*}| = d-1 = |\AA_0^{k^*}|.
}
Furthermore, \eqref{key_implication} gives the existence of some $j \neq k^*$ such that $a^{k^*}\notin N(b_{j})$.
Now take any $i\neq j$, $1 \leq i\leq d-1$.
Because $|\AA_0^{k^*}|=d-1$, $\AA_0$ necessarily contains $(a_i,a^{k^*})$.
Moreover, $a_i$ is not equal to $b_j$ because of \eqref{label_1}, and $a^{k^*}$ is not equal to $b_j$ because of \eqref{label_2}.
Consequently, for every $k=1,\dots,d$,
\eq{
(a_i,a^{k^*})\vdash (b_j,b_j^k) \quad \implies \quad (b_j,b_j^k)\in\cmp(\AA_0).
}
In particular, $(b_j,b_j^j)\in\cmp(\AA_0)$.
Now notice that because $\BB^j$ was defined via \eqref{B_0}, \eqref{B_1}, or \eqref{B_2}, we currently have $(b_j,b_j^j) \notin \BB^j$.
Therefore, adding $(b_j,b_j^j)$ to $\BB^j$ results in
\eq{
|\BB^j| \geq |\AA_0^j| + 2.
}
This new relation, combined with \eqref{ineq_1}, \eqref{ineq_2}, and \eqref{ineq_3}, again leads to \eqref{sum_condition}. \\


{\it Case 3c: $|\AA_0^k|=d-1$ for more than one value of $k$.}
Consider any $k$ such that $|\AA_0^{k}|=d-1$; in particular, $\{a_1,\dots,a_{d-1}\} \subset N(a^k)$.
As in the previous case, if $a^k \notin N(b)$, then \eqref{not_neighbor_consequence} allows us to take
\eeq{ \label{best_scenario}
\BB^k = \{(b_j,b_j^k) : 1 \leq j \leq d\},
}
so that
\eeq{ \label{plus_1}
|\BB^k| = d = |\AA_0^k| + 1.
}
If instead $a^k \in N(b)$, meaning $a^k=b_k$ by \eqref{label_2}, then we take $\ell\neq k$ such that $|\AA_0^{\ell}| = d-1$.
Since $\{a_1,\dots,a_{d-1}\} \subset N(a^\ell)$, the two sets $N(a^k)$ and $N(a^\ell)$ have $d-1$ common elements.
Hence $a^k$ and $a^\ell$ are not themselves adjacent, for otherwise $G$ would contain a copy of $H_d$.
Now take any $i\neq k$, $1\leq i \leq d-1$.
Again because $|\AA_0^{\ell}| = d-1$, we have $(a_i,a^\ell)\in\AA_0$.
Moreover, $b_k$ is not equal to $a_i$ because of \eqref{label_1}, $b_k=a^k$ is clearly not equal to $a^\ell$, and we have just reasoned that $a^\ell$ is not adjacent to $a^k=b_k$.
Consequently,
\eq{
(a_i,a^\ell)\vdash (b_k,b_k^j) \quad \text{for any $j=1,\dots,d$} \quad &\stackrefp{more_than_1_consequence}{\implies} \quad (b_k,b_k^k) \in \cmp(\AA_0) \\
\quad &\stackref{more_than_1_consequence}{\implies} \quad \{(b_j,b_j^k) : 1 \leq j\leq d\} \subset \cmp(\AA_0^k).
}
Therefore, it is still possible to take $\BB^k$ as in \eqref{best_scenario}, making \eqref{plus_1} nonetheless valid.
The combination of \eqref{ineq_1}, \eqref{ineq_2}, and \eqref{plus_1} yields \eqref{sum_condition} once more.
\end{proof}

\subsection{Step 3: construction of coupling} \label{formal_construction}
We can now give a precise construction of the coupling outlined in Section~\ref{construction_outline}.
Alice's initial position $A_0$ can be any vertex; let $E_1$ be a uniformly random element of $N(A_0)$.
Bob's initial position $B_0$ can be $E_1$ or any vertex not belonging to $N(A_0)\cup\{A_0\}$.
We then make the following inductive assumptions: 
\begin{subequations}
\label{even_induction}
\begin{align}
\label{even_2}
&\phantom{t\equiv 0 \mod 3 \qquad \implies \qquad}B_{t}\neq A_t, \\ \label{even_3}
&\phantom{t\equiv 0 \mod 3 \qquad \implies \qquad}\rlap{$B_t \in N(A_t) \text{ only if }  B_{t}=E_{t+1}$,}\phantom{A_{t+1}\in N(B_t) \text{ only if }  A_{t+1}=E_{t+1}.}
\\[-1.5\baselineskip] \notag
&t\equiv 0 \mod 3 \qquad \implies \qquad 
 \end{align}
 \end{subequations}
 \begin{subequations}
 \label{odd_induction}
\begin{align}
\label{odd_2}
&\phantom{t\equiv 1 \mod 3 \qquad \implies \qquad}A_{t+1}\neq B_t, \\ \label{odd_3}
&\phantom{t\equiv 1 \mod 3 \qquad \implies \qquad}A_{t+1}\in N(B_t) \text{ only if }  A_{t+1}=E_{t+1}.
\\[-1.5\baselineskip] \notag
&t\equiv 1 \mod 3 \qquad \implies \qquad \\[0.3\baselineskip] \notag
 \end{align}
 \end{subequations}
 We will show that \eqref{even_induction} with $t = 3q$ allows us to define $A_{3q+1},A_{3q+2},B_{3q+1},E_{3q+2}$ such that \eqref{odd_induction} holds with $t = 3q+1$.
 In turn, \eqref{odd_induction} with $t=3q+1$ will lead to $B_{3q+2},B_{3q+3},A_{3q+3}$, $E_{3q+4}$ satisfying \eqref{even_induction} with $t = 3q+3$.
 In this way, we will be able to inductively define $A_t$ and $B_t$ for all $t\geq0$.
 
 \begin{remark}
The prospective algorithm just described only defines $E_t$ for values of $t$ which are not a multiple of $3$.
This shall not concern us, however, since we are ultimately interested in only the process $(A_t,B_t)_{t\geq0}$.
 \end{remark}
 
Given \eqref{even_induction}, suppose $t=3q$, and $A_t=a$, $B_t = b$, $E_{t+1}=e$.
Let us fix enumerations $N(a)\setminus\{e\}=\{a_1,\dots,a_{d-1}\}$ and $N(b) = \{b_1,\dots,b_d\}$, as well as
\eq{
N(a_i) = \{a_i^1,\dots,a_i^d\},\ 1 \leq i \leq d-1, \qquad
N(b_j) = \{b_j^1,\dots,b_j^d\},\ 1 \leq j \leq d-1.
}
Using this notation, we recall the following definitions from Section~\ref{pair_neighbors}:
\eq{
\AA = \{(a_i,a_i^k) : 1 \leq i \leq d-1,\, 1\leq k \leq d\}, \qquad \BB = \{(b_j,b_j^\ell) : 1 \leq j\leq d,\, 1\leq \ell\leq d\}.
}
Consider the bipartite graph with vertex parts $(V_a,V_b)$ and edge set $\EE$ given as follows.
Let $V_a$ be the multiset in which every element of $\AA$ appears $d$ times, and let $V_b$ be the multiset in which every element of $\BB$ appears $d-1$ times.
Then we say $\EE$ contains edges between all instances of $(a_i,a_i^k)$ and $(b_j,b_j^\ell)$ whenever $(a_i,a_i^k)\vdash(b_j,b_j^\ell)$.

Let $\wt \AA_0$ be any sub-multiset of $V_a$.
Take $\wt\BB_0$ to be the sub-multiset of $V_b$ consisting of vertices adjacent to some element of $\wt\AA_0$.
If $\AA_0$ denotes the subset of $(a_i,a_i^k)\in\AA$ appearing at least once in $\wt\AA_0$, then $\wt\BB_0$ is precisely the multiset in which every element of $\cmp(\AA_0)$ appears $d-1$ times.
Clearly $|\AA_0|\geq |\wt\AA_0|/d$, and so Lemma~\ref{hall_prep_1}---which applies because of \eqref{even_induction}---gives
\eq{
|\wt\BB_0| = (d-1)|\cmp(\AA_0)| \geq d|\AA_0| \geq |\wt \AA_0|.
}
Therefore, by Hall's Marriage Theorem, $\EE$ contains a perfect matching between $V_a$ and $V_b$.
For each $i\in\{1,\dots,d-1\}$ and $j,k,\ell\in\{1,\dots,d\}$, let $m(i,j,k,\ell)$ denote the number of edges between $(a_i,a_i^k)$ and $(b_j,b_j^\ell)$ in this matching.
By definition of $V_a$ and $V_b$, we have
\eeq{ \label{total_sum}
\sum_{i=1}^{d-1}\sum_{j,k,\ell =1}^d m(i,j,k,\ell) &= |V_a| = |V_b| = d^2(d-1).
}
as well as
\begin{align}
\sum_{j=1}^{d}\sum_{\ell=1}^d m(i,j,k,\ell) &= \rlap{$d$}\phantom{d-1} \quad \text{for every $i\in\{1,\dots,d-1\}$, $k\in\{1,\dots,d\}$},\label{partial_1} \\
\sum_{i=1}^{d-1}\sum_{k=1}^d m(i,j,k,\ell) &= d-1 \quad \text{for every $j,\ell\in\{1,\dots,d\}$.}\label{partial_2}
\end{align}
Trivial consequences of \eqref{partial_1} and \eqref{partial_2} are
\begin{align}
\sum_{j=1}^d \sum_{k=1}^{d} \sum_{\ell=1}^d m(i,j,k,\ell) &= \rlap{$d^2$}\phantom{d(d-1)} \quad \text{for every $i\in\{1,\dots,d-1\}$}, \label{partial_3} \\
 \sum_{i=1}^{d-1} \sum_{k=1}^d \sum_{\ell=1}^d m(i,j,k,\ell) &= d(d-1) \quad \text{for every $j\in\{1,\dots,d\}$}. \label{partial_4}
 \end{align}


We can now couple $(A_{t+1},A_{t+2})$ and $(B_{t+1},E_{t+2})$ by first defining random variables $I\in\{1,\dots,d-1\}$ and $J,K,L\in\{1,\dots,d\}$, subject to
\eeq{ \label{couple_indices}
&\P\givenp{I=i,J=j,K=k,L=\ell}{A_{[0,t]},B_{[0,t]},E_{t+1}} \\
&= \P\givenp{I=i,J=j,K=k,L=\ell}{A_t=a,B_t=b,E_{t+1}=e} \\
&= \frac{m(i,j,k,\ell)}{d^2(d-1)}, \quad i\in\{1,\dots,d-1\},\, j,k,\ell \in \{1,\dots,d\}.
}
In light of \eqref{total_sum}, the above display prescribes a well-defined joint law for $(I,J,K,L)$.
In particular, we almost surely have $m(I,J,K,L) > 0$, which implies $(a_I,a_I^K) \vdash (b_J,b_J^L)$ by the definition of $\EE$.
We thus set
\eq{
(A_{t+1},A_{t+2}) &= (a_I,a_I^K), \qquad
(B_{t+1},E_{t+2}) = (b_J,b_J^L), 
}
noting that $A_{t+1} = a_I \neq e = E_{t+1}$.
By design we have $(A_{t+1},A_{t+2})\vdash(B_{t+1},E_{t+2})$, which shows
\eeq{
\label{avoid_1}
B_{t+1} \neq A_{t+1},A_{t+2} \quad \text{for $t = 3q$},
}
as well as \eqref{odd_induction} with $t = 3q+1$.
Furthermore,
\eeq{
\label{avoid_2}
A_{t+1} \neq E_{t+1} \quad \stackref{even_3}{\implies} \quad B_t \neq A_{t+1}  \quad \text{for $t=3q$}.
}


If $t=3q+1$, then we follow the same procedure but with Alice and Bob exchanging roles.
That is, supposing $B_t = a$, $A_{t+1} = b$, and $E_{t+1}=e$, we set
\eq{
(B_{t+1},B_{t+2}) = (a_I,a_I^K), \qquad (A_{t+2},E_{t+3}) = (b_J,b_J^L),
}
noting that $B_{t+1}=a_I \neq e = E_{t+1}$.
In this case, $(B_{t+1},B_{t+2}) \vdash(A_{t+2},E_{t+3})$ implies
\eeq{
\label{avoid_3}
B_{t+1},B_{t+2} \neq A_{t+2} \quad \text{for $t=3q+1$},
}
as well as \eqref{even_induction} for $t=3q+3$.
Furthermore,
\eeq{
\label{avoid_4}
B_{t+1} \neq E_{t+1} \quad \stackref{odd_3}{\implies} \quad B_{t+1} \neq A_{t+1}  \quad \text{for $t=3q+1$}.
}

\subsection{Step 4: verification of necessary properties} \label{srw_verification}
We need to check that $(A_t,B_t)_{t\geq0}$ satisfies the two conditions of Definition~\ref{ac_def}.
The avoidance property is a straightforward consequence of our construction:
\eq{
t\equiv0\mod3 \quad &\implies \quad B_t \stackref{avoid_3}{\neq} A_t \quad \text{and} \quad B_t \stackref{avoid_2}{\neq} A_{t+1}, \\
t\equiv1\mod3 \quad &\implies \quad B_t \stackref{avoid_1}{\neq} A_t,A_{t+1}, 
\\
t\equiv2\mod3 \quad &\implies \quad B_t \stackref{avoid_4}{\neq} A_t \quad \text{and} \quad B_t \stackref{avoid_3}{\neq} A_{t+1}. \\
}
To verify faithfulness, say in the case $t=3q$, we need to check the following identities:
\begin{subequations}
\label{faith_conditions}
\begin{align}
\P\givenp{A_{t+1}=a'}{A_{[0,t]}} &= \frac{1}{d}\one_{\{a'\in N(A_t)\}}, \label{faith_1} \\
\P\givenp{A_{t+2}=a''}{A_{[0,t+1]}} &= \frac{1}{d}\one_{\{a''\in N(A_{t+1})\}}, \label{faith_2} \\
\P\givenp{B_{t+1}=b'}{B_{[0,t]}} &= \frac{1}{d}\one_{\{b'\in N(B_t)\}}. \label{faith_3}
\end{align}
\end{subequations}
If $t=3q+1$, we would instead need to verify
\eq{
\P\givenp{B_{t+1}=b'}{B_{[0,t]}} &= \frac{1}{d}\one_{\{b'\in N(B_t)\}}, \\
\P\givenp{B_{t+2}=b''}{B_{[0,t+1]}} &= \frac{1}{d}\one_{\{b''\in N(B_{t+1})\}},  \\
\P\givenp{A_{t+2}=a'}{A_{[0,t+1]}} &= \frac{1}{d}\one_{\{a'\in N(A_{t+1})\}}.
}
Since the argument for this latter set of identities is analogous to the one for the former, we will just prove \eqref{faith_conditions}.
We need to establish, as an intermediate step, that the vertex $E_{t+1}$ to be avoided is uniform among the neighbors of whichever walker is next to move.
Moreover, this uniform distribution needs to be independent of that walker's history.

\begin{claim}
We have
\begin{subequations}
\begin{align}
t\equiv0\mod3 \quad &\implies \quad
\P\givenp{E_{t+1}=a'}{A_{[0,t]}} = \frac{1}{d}\one_{\{a'\in N(A_t)\}}, \label{even_1} \\
t\equiv1\mod3 \quad &\implies \quad
\P\givenp{E_{t+1}=b'}{B_{[0,t]}} = \frac{1}{d}\one_{\{b'\in N(B_t)\}}. \label{odd_1}
\end{align}
\end{subequations}
\end{claim}

\begin{proof}
Recall that $E_1$ was chosen so that \eqref{even_1} is true with $t=0$.
So we may assume by induction that \eqref{even_1} holds with $t=3q$, and then seek to prove \eqref{odd_1} with $t=3q+1$. 
Taking $t=3q$ and using the notation from Section~\ref{formal_construction} (in particular, $B_t=b$) we have
\eq{ 
&\P\givenp{E_{t+2}=b_j^\ell}{B_{[0,t]},B_{t+1} = b_j} \\
&= \sum_{a\in V\setminus\{b\}}\sum_{e\in N(a)}\P\givenp{L=\ell}{B_{[0,t]},J=j,A_t=a,E_{t+1}=e} \P\givenp{A_t=a,E_{t+1}=e}{B_{[0,t]},J=j}.
}
Recall from \eqref{couple_indices} that $(I,J,K,L)$ is conditionally independent of $A_{[0,t]}$ and $B_{[0,t]}$ given $A_t=a$, $B_t=b$, and $E_{t+1}=e$.
Therefore,
\eq{
&\P\givenp{L=\ell}{B_{[0,t]},J=j,A_t=a,E_{t+1}=e}
= \P\givenp{L=\ell}{J=j,A_t=a,B_t=b,E_{t+1}=e} \\
&= \sum_{i=1}^{d-1}\sum_{k=1}^{d} \P\givenp{I=i,J=j,K=k,L=\ell}{A_t=a,B_t=b,E_{t+1}=e} \\
&\phantom{=}\times\bigg(\sum_{\ell'=1}^d\sum_{i=1}^{d-1}\sum_{k=1}^{d} \P\givenp{I=i,J=j,K=k,L=\ell'}{A_t=a,B_t=b,E_{t+1}=e}\bigg)^{-1}  \\
&=\sum_{i=1}^{d-1}\sum_{k=1}^{d} \frac{m(i,j,k,\ell)}{d^2(d-1)}\bigg(\sum_{i=1}^{d-1}\sum_{k=1}^d\sum_{\ell'=1}^d \frac{m(i,j,k,\ell')}{d^2(d-1)}\bigg)^{-1}  
\stackrel{\mbox{\footnotesize\eqref{partial_2},\eqref{partial_4}}}{=}\frac{d-1}{d(d-1)} = \frac{1}{d}.
}
Using this computation in the previous display, we arrive at
\eq{
\P\givenp{E_{t+2}=b_j^\ell}{B_{[0,t]},B_{t+1} = b_j} = \frac{1}{d}\sum_{a\in V\setminus\{b\}}\sum_{e\in N(a)}\P\givenp{A_t=a,E_{t+1}=e}{B_{[0,t]},J=j} = \frac{1}{d},
}
thus proving \eqref{odd_1}.
The argument that \eqref{odd_1} with $t=3q+1$ implies \eqref{even_1} with $t=3q+3$ is completely analogous.
\end{proof}

We can now establish \eqref{faith_conditions}.
Let us continue using the notation of Section~\ref{formal_construction}.
First consider any $a'\in N(A_t)$, where $A_t=a$.
Because $I$ is conditionally independent of $A_{[0,t]}$ given $A_t,B_t,E_{t+1}$, we have
\eq{
&\P\givenp{A_{t+1}=a'}{A_{[0,t]}}
= \sum_{e\in N(a)\setminus\{a'\}}\sum_{b\in \{e\} \cup V\setminus N(a)}\Big[\P\givenp{B_t=b,E_{t+1}=e}{A_{[0,t]}}\\
&\phantom{\P\givenp{A_{t+1}=a'}{A_{[0,t]}}=}\times\sum_{i=1}^{d-1}\one_{\{a_i = a'\}}\P\givenp{I=i}{A_t=a,B_t=b,E_{t+1}=e}\Big] \\
&{\stackrel{{\mbox{\footnotesize\eqref{partial_3},\eqref{couple_indices}}}}{=}}
\frac{1}{d-1}\sum_{e\in N(a)\setminus\{a'\}}\sum_{b\in \{e\} \cup V\setminus N(a)}\P\givenp{B_t=b,E_{t+1}=e}{A_{[0,t]}} \\
&{\stackrel{\phantom{\mbox{\footnotesize\eqref{partial_3},\eqref{couple_indices}}}}{=}}\frac{1}{d-1}\sum_{e\in N(a)\setminus\{a'\}}\P\givenp{E_{t+1}=e}{A_{[0,t]}} 
\stackref{even_1}{=} \frac{1}{d-1}\sum_{e\in N(a)\setminus\{a'\}}\frac{1}{d} = \frac{1}{d}.
}
That is, \eqref{faith_1} holds.
We next verify \eqref{faith_2}.
For any $a''\in N(A_{t+1})$, where $A_{t+1}=a'$, we compute 
\eq{
&\P\givenp{A_{t+1}=a''}{A_{[0,t+1]}} \\
&\stackrel{\phantom{\mbox{\footnotesize\eqref{partial_1},\eqref{partial_3}}}}{=}\sum_{e\in N(a)\setminus\{a'\}}\sum_{b\in \{e\}\cup V\setminus N(a)}\Big[\P\givenp{B_t=b,E_{t+1}=e}{A_{[0,t+1]}} \\
&\phantom{\stackrel{\mbox{\footnotesize\eqref{partial_1},\eqref{partial_3}}}{=}}\times\sum_{i=1}^{d-1}\sum_{k=1}^d\one_{\{a_i=a',a_i^k=a''\}}\P\givenp{K=k}{A_t=a,B_t=b,E_{t+1}=e,I=i}\Big] \\
&\stackrel{\phantom{\mbox{\footnotesize\eqref{partial_1},\eqref{partial_3}}}}{=} \sum_{e\in N(a)\setminus\{a'\}}\sum_{b\in \{e\}\cup V\setminus N(a)}\bigg[\P\givenp{B_t=b,E_{t+1}=e}{A_{[0,t+1]}} \\
&\phantom{\stackrel{\mbox{\footnotesize\eqref{partial_1},\eqref{partial_3}}}{=}}\times\sum_{j,\ell=1}^{d}\frac{m(i,j,k,\ell)}{d^2(d-1)}\bigg(\sum_{k'=1}^{d}\sum_{j,\ell=1}^{d} \frac{m(i,j,k',\ell)}{d^2(d-1)}\bigg)^{-1}\bigg] \\
&\stackrel{\mbox{\footnotesize\eqref{partial_1},\eqref{partial_3}}}{=}\frac{1}{d}\sum_{e\in N(a)\setminus\{a'\}}\sum_{b\in \{e\}\cup V\setminus N(a)} \P\givenp{B_t=b,E_{t+1}=e}{A_{[0,t+1]}} = \frac{1}{d}.
}
Finally, for \eqref{faith_3}, consider any $b'\in N(B_t)$ where $B_t=b$.
Because $J$ is conditionally independent of $B_{[0,t]}$ given $A_t,B_t,E_{t+1}$, we have
\eq{
&\P\givenp{B_{t+1}=b'}{B_{[0,t]}} \\
&= \sum_{a\in V\setminus\{b\}}\sum_{e\in N(a)} \sum_{j=1}^d\one_{\{b'=b_j\}}\P\givenp{J = j}{A_t=a,B_t=b,E_{t+1}=e}
\P\givenp{A_t=a,E_{t+1}=e}{B_{[0,t]}} \\
&\stackrel{\mbox{\footnotesize\eqref{couple_indices}\eqref{partial_4}}}{=} \frac{1}{d}\sum_{a\in V\setminus\{b\}}\sum_{e\in N(a)}\P\givenp{A_t=a,E_{t+1}=e}{B_{[0,t]}} = \frac{1}{d}.
}
We have now have proved \eqref{faith_conditions}, thereby completing the construction.

\section{Proof of Theorem~\ref{square_free_thm}: square-free graphs} \label{square_free}

\setcounter{subsection}{-1}
\subsection{Outline of coupling} \label{construction_square_free}
Here the coupling strategy incorporates features from both Section~\ref{3_regular} and Section~\ref{d_regular}.
Specifically, we will describe in the square-free case certain local couplings satisfying conditions \hyperref[condition_1]{(i)}--\hyperref[condition_3]{(iii)} from Section~\ref{3_outline}.
To prove the existence of said couplings, we will again appeal to Hall's Marriage Theorem as in Section~\ref{d_regular}. 
Fortunately, each of these two tasks is more straightforward for square-free graphs than for regular graphs.
In the first, we can always take $T=1$; in the second, the relevant combinatorics are significantly simpler.

Throughout the remainder of Section~\ref{square_free}, we assume
\begin{itemize}
\item Alice is at vertex $A_t = a$ and is next to move;
\item Bob is at vertex $B_t = b$, which is neither equal to $a$ nor a neighbor of $a$.
\end{itemize}
Denoting $|N(a)| = k$ and $|N(b)| = \ell$, let us also assume $k\geq\ell$; because we will always take $T=1$, the reverse scenario can be handled in a completely symmetric manner.
Recall that Theorem~\ref{square_free_thm} assumes $k,\ell\geq3$.

Our goal is to specify $A_{t+1}$ and $B_{t+1}$ such that \hyperref[condition_1]{(i)}--\hyperref[condition_3]{(iii)} from Section~\ref{3_regular} are satisfied.
This is accomplished in three steps.
Section~\ref{graph_prelim} records some trivial properties of square-free graphs.
These properties are then used in Section~\ref{combin_prelim} to prove a combinatorial lemma needed to construct the desired coupling.
Finally, Section~\ref{construction} sees to fruition the actual construction.
Once this coupling is realized, Theorem~\ref{square_free_thm} will follow from Proposition~\ref{conditions_to_coupling}.

\subsection{Step 1: graph theoretic preliminaries} \label{graph_prelim}

The critical properties of a square-free graph are the following.

\begin{lemma} \label{square_free_lemma}
Let $G = (V,E)$ be a square-free graph, and suppose $a,b\in V$ satisfy $b\notin\{a\}\cup N(a)$.
Then the following statements hold:
\begin{enumerate}[label=\textup{(\alph*)}]
\item For any $a'\in N(a)$, $|N(a')\cap N(b)|\leq1$.
\item For any $b'\in N(a)$, $|N(b')\cap N(a)|\leq1$.
\item $|N(a)\cap N(b)| \leq 1$.
\end{enumerate}
\end{lemma}

\begin{proof}
The roles of $a$ and $b$ are interchangeable, and so (b) follows from (a).
For (a), observe that if there were distinct vertices $b_1,b_2\in N(a')\cap N(b)$, then $a' \to b_1\to b \to b_2 \to a'$ would be a square.
Similarly for (c), the existence of distinct $c_1,c_2\in N(a)\cap N(b)$ would lead to the square $a\to c_1\to b\to c_2\to a$.
\end{proof}

\subsection{Step 2: combinatorics of compatible moves} \label{combin_prelim}
Since $T=1$, conditions \hyperref[condition_2]{(ii)} and \hyperref[condition_3]{(iii)} are equivalent to $B_{t+1} \notin \{A_{t+1}\} \cup N(A_{t+1})$.
Therefore, to construct a suitable coupling between $A_{t+1}$ and $B_{t+1}$, it will be useful to make the following definition.  
Given a subset $\NN\subset N(a)$, define the set
\eq{
\cmp(\NN) \coloneqq \big\{b' \in N(b)\, :\, \exists\, a'\in \NN,\, b'\notin\{a'\}\cup N(a')\big\}.
}
The following combinatorial lemma will allow us to prove existence of the desired coupling.

\begin{lemma} \label{hall_prep_2}
For any $\NN\subset N(a)$, we have $|\cmp(\NN)| \geq (\ell/k)|\NN|$.
\end{lemma}

\begin{proof}
We separately consider two possibilities. \\

\hypertarget{square_case_1}{{\bf Case 1: $N(a)\cap N(b)$ is empty.}}
Let us enumerate the elements of $N(a)$ and $N(b)$ in the following greedy way.
First choose $a_1\in N(a)$ and $b_1\in N(b)$ to be adjacent if possible.
Then repeat, selecting $a_2\in N(a)\setminus\{a_1\}$ and $b_2\in N(b)\setminus\{b_1\}$ to be adjacent if possible.
Continue until there are no further adjacencies between the remaining elements of $N(a)$ and $N(b)$, at which point these remaining elements can be labeled arbitrarily.
In any circumstance, Lemma~\ref{square_free_lemma}(a,b) implies $b_j \in \{a_i\} \cup N(a_i)$ only if $j = i$.
Hence
\eeq{ \label{two_implications}
|\NN| = 1 \quad &\implies \quad |\cmp(\NN)| \geq \ell-1 \quad \implies \quad  |\cmp(\NN)| \geq 2 > |\NN| \geq (\ell/k)|\NN|, \\
|\NN| \geq 2 \quad &\implies \quad |\cmp(\NN)| = \rlap{$\ell$}\phantom{\ell-1} \quad \implies \quad  |\cmp(\NN)|= (\ell/k)|N(a)| \geq (\ell/k)|\NN|.
}
\phantom{text}

{\bf Case 2: $N(a)\cap N(b)$ is nonempty.}
In this scenario, Lemma~\ref{square_free_lemma}(c) allows only $|N(a)\cap N(b)| = 1$.
Let us denote the single element of $N(a)\cap N(b)$ by $c=a_1=b_1$. \\

{\it Case 2a: $c$ is adjacent to neither $N(a)$ nor $N(b)$.}
Here we are assuming $N(a)\cap N(c) = N(b)\cap N(c) = \varnothing$.
We then enumerate the remaining elements of 
\eq{
N(a)\setminus\{c\} = \{a_2,\dots,a_k\} \quad \text{and} \quad N(b)\setminus\{c\} = \{b_2,\dots,b_\ell\}
} 
as in \hyperlink{square_case_1}{Case 1}.
That is, for $2\leq i \leq k$ and $2\leq j\leq\ell$, we have $b_j \in \{a_i\} \cup N(a_i)$ only if $j = i$.
Notice this statement also holds if $i$ or $j$ is equal to $1$, thanks to our assumption $N(a)\cap N(c) = N(b)\cap N(c) = \varnothing$.
Hence the implications in \eqref{two_implications} remain true, and so the lemma's claim remains valid. \\

\hypertarget{square_case_2b}{{\it Case 2b: $c$ is adjacent to both $N(a)$ and $N(b)$.}}
Now we are assuming $N(a)\cap N(c) = \{a_2\}$ and $N(b)\cap N(c) = \{b_2\}$, where $a_2$ and $b_2$ are necessarily distinct by Lemma~\ref{square_free_lemma}(c).
Moreover, $N(a_2) \cap N(b) = \{c\}$ by Lemma~\ref{square_free_lemma}(a), and $N(b_2) \cap N(a) = \{c\}$ by Lemma~\ref{square_free_lemma}(b).
We can now enumerate the remaining elements of $N(a)$ and $N(b)$ as in \hyperlink{square_case_1}{Case 1}, so that
\eq{
a_3 \in N(b_j) \quad &\implies \quad j = 3, \\
&\hspace{2.6ex}\vdots \\
a_k \in N(b_j) \quad &\implies \quad j = k.
}
Together, these observations imply
\eq{ 
c=a_1 \in \NN \quad &\implies \quad b_j \in \cmp(\NN) \text{ for every $j \neq 1,2$}, \\
a_2 \in \NN \quad &\implies \quad b_j \in \cmp(\NN) \text{ for every $j \neq 1$}, \\
a_i \in \NN,\, i\geq3 \quad &\implies \quad b_j \in \cmp(\NN) \text{ for every $j \neq i$}.
}
A slightly weaker form of \eqref{two_implications}, but still satisfying the claim, readily follows:
\eeq{ \label{three_cases}
|\NN| = 1 \quad &\implies \quad |\cmp(\NN)| \geq  \ell-2 \quad \implies \quad |\cmp(\NN)| \geq 1 = |\NN| \geq (\ell/k)|\NN|, \\
|\NN| = 2 \quad &\implies \quad |\cmp(\NN)| \geq \rlap{$\ell-1$}\phantom{\ell-2} \quad \implies \quad |\cmp(\NN)| \geq 2 = |\NN| \geq (\ell/k)|\NN|,\\
|\NN| \geq 3 \quad &\implies \quad |\cmp(\NN)| = \rlap{$\ell$}\phantom{\ell-2} \quad \implies \quad |\cmp(\NN)| = (\ell/k)|N(a)| \geq (\ell/k)|\NN|.
}
\phantom{text}

{\it Case 2c: $c$ is adjacent to only $N(a)$.}
Here $N(a) \cap N(c) = \{a_2\}$ but $N(b)\cap N(c) = \varnothing$.
As in \hyperlink{square_case_2b}{Case 2b}, the first equality implies $N(a_2) \cap N(b) = \{c\}$.
We next enumerate the remaining elements of $N(a)$ and $N(b)$ in a greedy fashion similar to before, so that
\eq{
a_3 \in N(b_j) \quad &\implies \quad j = 2, \\
&\hspace{2.6ex}\vdots \\
a_k \in N(b_j) \quad &\implies \quad j = k-1.
}
Notice that if $j = k$, then $b_j$ is adjacent to no element of $N(a)$.
We now have 
\eq{
c=a_1 \in \NN \quad &\implies \quad b_j \in \cmp(\NN) \text{ for every $j \neq 1$}, \\
a_2 \in \NN \quad &\implies \quad b_j \in \cmp(\NN) \text{ for every $j \neq 1$}, \\
a_i \in \NN,\, i\geq3 \quad &\implies \quad b_j \in \cmp(\NN) \text{ for every $j \neq i-1$}.
}
It is easy to check that \eqref{three_cases} remains true. \\

{\it Case 2d: $c$ is adjacent to only $N(b)$.}
In this final case, we have $N(a) \cap N(c) = \varnothing$ and $N(b)\cap N(c) = \{b_2\}$.
The greedy enumeration now leads to
\eq{
a_2 \in N(b_j) \quad &\implies \quad j = 3, \\
&\hspace{2.6ex}\vdots \\
a_{k-1} \in N(b_j) \quad &\implies \quad j = k.
}
Given that $\ell\leq k$, we are left to conclude $a_{k} \notin N(b_j)$ for every $j=1,\dots,\ell$.
Hence
\eq{
c=a_1 \in \NN \quad &\implies \quad b_j \in \cmp(\NN) \text{ for every $j \neq 1,2$}, \\
a_i \in \NN,\, i\geq2 \quad &\implies \quad b_j \in \cmp(\NN) \text{ for every $j \neq i+1$},
}
from which one deduces
\eq{
|\NN| = 1 \quad &\implies \quad |\cmp(\NN)| \geq \ell-2 \quad \implies \quad |\cmp(\NN)| \geq 1 = |\NN| \geq (\ell/k)|\NN|, \\
|\NN| \geq 2 \quad &\implies \quad |\cmp(\NN)| \geq \rlap{$\ell$}\phantom{\ell-2} \quad \implies \quad |\cmp(\NN)| = (\ell/k)|N(a)| \geq (\ell/k)|\NN|.
}
Now all cases have been handled, and so the proof is complete.
\end{proof}

\subsection{Step 3: construction of the coupling} \label{construction}
Recall that $k = |N(a)| \geq |N(b)| = \ell$.
Let us fix enumerations $N(a) = \{a_1,\dots,a_k\}$ and $N(b) = \{b_1,\dots,b_\ell\}$.
Consider the bipartite graph with vertex parts $(V_a,V_b)$ and edge set $\EE$ given as follows.
Let $V_a$ be the multiset in which every element of $N(a)$ appears $\ell$ times, and let $V_b$ be the multiset in which every element of $N(b)$ appears $k$ times.
We suppose that $\EE$ contains edges between all instances of $a_i$ and $b_j$ whenever $b_j \notin \{a_i\} \cup N(a_i)$.

Now let $\wt \NN_a$ be any sub-multiset of $V_a$.
Let $\wt\NN_b$ be the sub-multiset of $V_b$ consisting of vertices adjacent to some element of $\wt\NN_a$.
If $\NN_a$ denotes the subset of $a_i\in N(a)$ appearing at least once in $\wt\NN_a$, then $\wt\NN_b$ is the multiset in which every element of $\cmp(\NN_a)$ appears $k$ times.
Clearly $|\NN_a|\geq|\wt\NN_a|/\ell$, and so Lemma~\ref{hall_prep_2} gives
\eq{
|\wt\NN_b| = k|\cmp(\NN_a)| \geq \ell|\NN_a| \geq |\wt \NN_a|.
}
Therefore, by Hall's Marriage theorem, $\EE$ contains a perfect matching between $V_a$ and $V_b$.
For $i=1,\dots,k$ and $j=1,\dots,\ell$, let $m(i,j)$ denote the number of edges between $a_i$ and $b_j$ in this matching.
By definition of $V_a$ and $V_b$, we have
\begin{align}
\label{uniform_1}
\sum_{j=1}^\ell m(i,j) &= \ell \quad \text{for every $i=1,\dots,k$}, \\
\label{uniform_2}
\sum_{i=1}^k m(i,j) &= k \quad \text{for every $j=1,\dots,\ell$}.
\end{align}

We can now couple $A_{t+1}$ and $B_{t+1}$ as follows.
Let $I$ be a uniformly random element of $\{1,\dots,k\}$, independent of $A_{[0,t]}$ and $B_{[0,t]}$, and then sample a random $J\in \{1,\dots,\ell\}$ according to
\eeq{ \label{J_def}
\P\givenp{J = j}{A_{[0,t]},B_{[0,t]},I} 
= \P\givenp{J = j}{A_t=a,B_t=b,I} 
= \frac{m(I,j)}{\ell}, \quad 1 \leq j\leq\ell.
}
Because of \eqref{uniform_1}, the above display prescribes a well-defined law for $J$.
Now, given the variables $I$ and $J$, we set
\eq{
A_{t+1} = a_I, \qquad B_{t+1} = b_J. 
}

\subsection{Step 4: verification of necessary properties}
Clearly $A_{t+1}$ is a uniformly random element of $N(a)$; moreover, $I$ is independent of $A_{[0,t]}$ so that \eqref{srw_condition_a} is satisfied.
Meanwhile, we claim $B_{t+1}$ is a uniformly random element of $N(B_t)$.
Indeed, observe from \eqref{J_def} that $J$ is conditionally independent of $A_{[0,t]},B_{[0,t]}$ given $A_t=a$, $B_t=b$, and $I$.
Furthermore, $I$ is completely independent of $A_{[0,t]},B_{[0,t]}$.
Consequently,
\eq{
&\P\givenp{B_{t+1}=b_j}{B_{[0,t]}} \\
&\stackrefp{uniform_2}{=}  \sum_{a\in V\setminus\{b\}}\sum_{i=1}^k\P\givenp{J=j}{A_t=a,B_t=b,I=i}\cdot\P(I=i)\cdot\P\givenp{A_t=a}{B_{[0,t]}} \\
&\stackrefp{uniform_2}{=}  \sum_{a\in V\setminus\{b\}}\sum_{i=1}^k \frac{m(i,j)}{\ell}\cdot\frac{1}{k}\cdot\P\givenp{A_t=a}{B_{[0,t]}} \\
&\stackref{uniform_2}{=} \frac{1}{\ell}\sum_{a\in V\setminus\{b\}}\P\givenp{A_t=a}{B_{[0,t]}} = \frac{1}{\ell}.
}
We have thus shown that \eqref{srw_condition_b} also holds, thereby verifying condition \hyperref[condition_1]{(i)}.
Meanwhile, conditions \hyperref[condition_2]{(ii)} and \hyperref[condition_3]{(iii)} follow by induction from the following chain of implications:
\eq{
\P\givenp{A_{t+1}=a_i,B_{t+1}=b_j}{A_t=a,B_t=b} > 0  &\iff 
m(i,j) > 0 \\ 
&\hspace{0.5ex}\implies  \{a_i,b_j\}\in\EE  \implies  b_{j} \notin \{a_i\}\cup N(a_i).
}
This completes the construction.

\section{Acknowledgments}
We are grateful to Omer Angel for suggesting the problem of constructing an avoidance coupling on regular graphs.
We thank the referees for their useful suggestions and edits.
This work was conducted in part during a visit of E.B. to the Research Institute of Mathematical Sciences at NYU Shanghai; E.B. thanks the Institute for its wonderful hospitality. 
E.B. was partially supported by NSF grant DMS-1902734.

\appendix

\section{Proof of Corollary~\ref{main_cor}} \label{appendix}
Let $\G(n,d)$ denote a uniformly random $d$-regular graph on $n$ vertices, if such a graph exists.
Let $\G_\cc(n,d)$ denote a uniformly random \textit{connected} $d$-regular graph on $n$ vertices, again if such a graph exists.
To prove Corollary~\ref{main_cor} given Theorem~\ref{main_thm}, we wish to show that
\eeq{ \label{cor_WTS}
\lim_{n\to\infty} \P\big(\wt H_3 \in \G_\cc(n,3)\big) = 0, \qquad \lim_{n\to\infty} \P\big(H_d\in \G_\cc(n,d)\big) = 0, \quad d\geq4. 
}
First we recall from \cite[Section 7.6]{bollobas01} that 
\eq{
\lim_{n\to\infty}\P(\G(n,d) \text{ is connected}) = 1 \quad \text{for all $d\geq3$}.
}
(In fact, $\G(n,d)$ is asymptotically $d$-connected.)
Consequently, for any sequence of events $\{E_n\}_{n\geq1}$ (\textit{i.e.}~$E_n$ is a subset of graphs on $n$ vertices), we have
\eq{
\lim_{n\to\infty}\P(\G(n,d) \in E_n) = 0 \quad \implies \quad \lim_{n\to\infty}\P(\G_\cc(n,d)\in E_n) = 0.
}
Therefore, to conclude \eqref{cor_WTS}, it suffices to show the same statements with $\G(n,d)$ replacing $\G_\cc(n,d)$:
\eeq{ \label{cor_WTS_2}
\lim_{n\to\infty} \P\big(\wt H_3 \in \G(n,3)\big) = 0, \qquad \lim_{n\to\infty} \P\big(H_d\in \G(n,d)\big) = 0, \quad d\geq4. 
}

To prove \eqref{cor_WTS_2}, we pass to a third random graph model.
When $nd$ is even, let $\G^*(n,d)$ denote the $d$-regular configuration model on $n$ vertices.
That is, choose a uniformly random partition $\PP(n,d)$ of the set $\{1,\dots,n\}\times\{1,\dots,d\}$ into $nd/2$ pairs; for each pair $\{(i,k),(j,\ell)\}$ in said partition, we include an edge between $i$ and $j$.
This forms a random multigraph (possibly with loops) on the vertex set $\{1,\dots,n\}$, which we denote $\G^*(n,d)$.
It is well-known, \textit{e.g.}~\cite[Theorem 9.9]{janson-luczak-rucinski00}, that for any sequence of events $\{E_n\}_{n\geq1}$ (now $E_n$ is a subset of \textit{multi}graphs on $n$ vertices), we have
\eeq{ \label{cor_have}
\lim_{n\to\infty} \P(\G^*(n,d)\in E_n) = 0 \quad \implies \quad \lim_{n\to\infty} \P(\G(n,d)\in E_n) = 0.
}
(This is because $\G^*(n,d)$ conditioned to be a simple graph is equal in law to $\G(n,d)$, and this conditioning event occurs with non-vanishing probability.)
Now \eqref{cor_WTS_2} is implied by \eqref{cor_have} together with the following lemma.

\begin{lemma}
If $H$ is a (simple) graph with $n_0$ vertices and $m>n_0$ edges, then
\eq{
\lim_{n\to\infty}\P\big(H\subset\G^*(n,d)\big) = 0.
}
\end{lemma}

\begin{proof}
Let us write the vertex set of $\G^*(n,d)$ as $\{1,\dots,n\}$, where we assume $n>2m$.
For $1\leq i<j\leq n$, $\{i,j\}$ is an edge in $\G^*(n,d)$ precisely when $\PP(n,d)$ contains a pair of the form $\{(i,k),(j,\ell)\}$, where $1\leq k,\ell\leq d$.
Since each of the ${nd \choose 2}$ possible pairs is equally likely to belong to the partition, and the partition contains $nd/2$ pairs, we have
\eq{
\P\big(\{i,j\}\in E(\G^*(n,d))\big) \leq \sum_{k,\ell=1}^d \P\big(\{(i,k),(j,\ell)\}\in\PP(n,d)\big) = \frac{d^2}{{nd \choose 2}}\frac{nd}{2} = \frac{d^2}{nd-1}.
}
In fact, we can generalize this computation as follows.

We know any $(i,k)$ must be matched with some $(j,\ell)$.
If $\{i_1,j_1\},\dots,\{i_q,j_q\}$ are already known to be edges in $\G^*(n,d)$, 
then there are at least $(n-2q)d$ elements of $\{1,\dots,n\}\times\{1,\dots,d\}$ remaining unmatched. 
Consequently, for any $i \neq j$ such that $\{i,j\} \neq \{i_p,j_p\}$ for $p=1,\dots,q$, we have
\eq{
&\P\givenp[\big]{\{i,j\}\in E(\G^*(n,d))}{\{i_1,j_1\},\dots,\{i_q,j_q\}\in E(\G^*(n,d))} \\
&\leq \sum_{k,\ell=1}^d \P\givenp[\big]{\{(i,k),(j,\ell)\}\in\PP(n,d)}{\{i_1,j_1\},\dots,\{i_q,j_q\}\in E(\G^*(n,d))} 
\leq \frac{d^2}{(n-2q)d-1}.
}
It follows that
\eq{
\P\big(H\subset\G^*(n,d)\big) \leq {n \choose n_0}\prod_{q=0}^{m-1}\frac{d^2}{(n-2q)d-1} \leq n^{n_0}\Big(\frac{d}{n-2m}\Big)^m.
}
As $m>n_0$, the above quantity vanishes as $n\to\infty$.
\end{proof}

\bibliography{avoidance_couplings}

\begin{thebibliography}{10}

\bibitem{angel-holroyd-martin-wilson-winkler13}
{\sc Angel, O., Holroyd, A.~E., Martin, J., Wilson, D.~B., and Winkler, P.}
\newblock Avoidance coupling.
\newblock {\em Electron. Commun. Probab. 18\/} (2013), no. 58, 13.

\bibitem{balister-bollobas-stacey00}
{\sc Balister, P.~N., Bollob\'as, B., and Stacey, A.~M.}
\newblock Dependent percolation in two dimensions.
\newblock {\em Probab. Theory Related Fields 117}, 4 (2000), 495--513.

\bibitem{basu-sidoravicius-sly19}
{\sc Basu, R., Sidoravicius, V., and Sly, A.}
\newblock Scheduling of non-colliding random walks.
\newblock In {\em Sojourns in Probability Theory and Statistical Physics -
  III\/} (Singapore, 2019), V.~Sidoravicius, Ed., Springer Singapore,
  pp.~90--137.

\bibitem{bates-sauermann19}
{\sc Bates, E., and Sauermann, L.}
\newblock An upper bound on the size of avoidance couplings.
\newblock {\em Combin. Probab. Comput. 28}, 3 (2019), 325--334.

\bibitem{bollobas01}
{\sc Bollob\'{a}s, B.}
\newblock {\em Random graphs}, second~ed., vol.~73 of {\em Cambridge Studies in
  Advanced Mathematics}.
\newblock Cambridge University Press, Cambridge, 2001.

\bibitem{brightwell-winkler09}
{\sc Brightwell, G.~R., and Winkler, P.}
\newblock Submodular percolation.
\newblock {\em SIAM J. Discrete Math. 23}, 3 (2009), 1149--1178.

\bibitem{coppersmith-tetali-winkler93}
{\sc Coppersmith, D., Tetali, P., and Winkler, P.}
\newblock Collisions among random walks on a graph.
\newblock {\em SIAM J. Discrete Math. 6}, 3 (1993), 363--374.

\bibitem{feldheim17}
{\sc Feldheim, O.~N.}
\newblock Monotonicity of avoidance coupling on {$K_N$}.
\newblock {\em Combin. Probab. Comput. 26}, 1 (2017), 16--23.

\bibitem{gacs10}
{\sc G\'{a}cs, P.}
\newblock The clairvoyant demon has a hard task.
\newblock {\em Combin. Probab. Comput. 9}, 5 (2000), 421--424.

\bibitem{gacs11}
{\sc G{\'a}cs, P.}
\newblock Clairvoyant scheduling of random walks.
\newblock {\em Random Structures Algorithms 39}, 4 (2011), 413--485.

\bibitem{grimmett10}
{\sc Grimmett, G.}
\newblock Three problems for the clairvoyant demon.
\newblock In {\em Probability and mathematical genetics}, vol.~378 of {\em
  London Math. Soc. Lecture Note Ser.} Cambridge Univ. Press, Cambridge, 2010,
  pp.~380--396.

\bibitem{infeld16}
{\sc Infeld, E.~J.}
\newblock {\em Uniform avoidance coupling, design of anonymity systems and
  matching theory}.
\newblock PhD thesis, 2016.
\newblock Thesis (Ph.D.)--Dartmouth College.

\bibitem{janson-luczak-rucinski00}
{\sc Janson, S., \L~uczak, T., and Rucinski, A.}
\newblock {\em Random graphs}.
\newblock Wiley-Interscience Series in Discrete Mathematics and Optimization.
  Wiley-Interscience, New York, 2000.

\bibitem{levin-peres17}
{\sc Levin, D.~A., and Peres, Y.}
\newblock {\em Markov chains and mixing times}.
\newblock American Mathematical Society, Providence, RI, 2017.
\newblock Second edition of [ MR2466937], With contributions by Elizabeth L.
  Wilmer, With a chapter on ``Coupling from the past'' by James G. Propp and
  David B. Wilson.

\bibitem{levin-peres-wilmer09}
{\sc Levin, D.~A., Peres, Y., and Wilmer, E.~L.}
\newblock {\em Markov chains and mixing times}.
\newblock American Mathematical Society, Providence, RI, 2009.
\newblock With a chapter by James G. Propp and David B. Wilson.

\bibitem{moseman-winkler08}
{\sc Moseman, E.~R., and Winkler, P.}
\newblock On a form of coordinate percolation.
\newblock {\em Combin. Probab. Comput. 17}, 6 (2008), 837--845.

\bibitem{pete08}
{\sc Pete, G.}
\newblock Corner percolation on {$\Bbb Z^2$} and the square root of 17.
\newblock {\em Ann. Probab. 36}, 5 (2008), 1711--1747.

\bibitem{tetali-winkler93}
{\sc Tetali, P., and Winkler, P.}
\newblock Simultaneous reversible {M}arkov chains.
\newblock In {\em Combinatorics, {P}aul {E}rd{\H o}s is eighty, {V}ol.\ 1},
  Bolyai Soc. Math. Stud. J\'anos Bolyai Math. Soc., Budapest, 1993,
  pp.~433--451.

\bibitem{winkler00}
{\sc Winkler, P.}
\newblock Dependent percolation and colliding random walks.
\newblock {\em Random Structures Algorithms 16}, 1 (2000), 58--84.

\end{thebibliography}

\end{document}